\documentclass[11pt]{amsart}
\usepackage{amsmath,amssymb,mathrsfs,color}
\usepackage[title]{appendix}
\usepackage{amssymb}
\allowdisplaybreaks

\let\cal=\mathcal
\def\N{{\mathbb N}}

\def\R{{\mathbb R}}
\def\P{{\mathbb P}}
\def\E{{\mathbb E}}
\def\T{{\mathbb T}}

\def\eps{\epsilon}
\newtheorem{thm}{Theorem}[section]
\newtheorem{cor}[thm]{Corollary}
\newtheorem{lem}[thm]{Lemma}
\newtheorem{prop}[thm]{Proposition}

\theoremstyle{definition}

\theoremstyle{remark}
\newtheorem{rem}[thm]{Remark}

\numberwithin{equation}{section}

\newcommand{\rmd}{{\rm d}}

\allowdisplaybreaks

\begin{document}

\title[Wasserstein convergence rates in the invariance principle]{Wasserstein convergence rates in the invariance principle for sequential dynamical systems}

\author{Zhenxin Liu}
\address{Z. Liu: School of Mathematical Sciences, Dalian University of Technology, Dalian
116024, P. R. China}
\email{zxliu@dlut.edu.cn}

\author{Zhe Wang$^{\ast}$}
\address{Z. Wang: School of Mathematical Sciences, Dalian University of Technology, Dalian
116024, P. R. China}
 \email{zwangmath@hotmail.com; wangz1126@mail.dlut.edu.cn}

\thanks{*Corresponding author}
\date{October 19, 2024}



\begin{abstract}
In this paper, we consider the convergence rate with respect to the Wasserstein distance in the invariance principle for sequential dynamical systems. We utilize and modify the techniques previously employed for stationary sequences to address our non-stationary case. Under certain assumptions, we can apply our result to a class of dynamical systems, including sequential $\beta_n$-transformations, piecewise uniformly expanding maps with additive noise in one-dimensional and multidimensional case, and so on.
\end{abstract}

\maketitle

\section{Introduction}
\setcounter{equation}{0}

There is considerable interest in the study of statistical properties for deterministic dynamical systems exhibiting hyperbolicity, wherein the same map is iterated all along the time. Due to the presence of an absolutely continuous invariant measure, the observable processes along the orbit become stationary. However,
in many physical applications, it is often the case that different maps are iterated randomly. This situation can be described as a (discrete) time-dependent dynamical system. Over the past few decades, there has been a growing interest in proving statistical properties for time-dependent dynamical systems, including sequential dynamical systems and random dynamical systems. Unlike the time-independent systems, time-dependent systems lack a universal invariant measure across all maps. Due to the lack of invariant measure and the fact that maps change with time, the processes are non-stationary, which causes some difficulties in study.

Sequential dynamical systems, as introduced by Berend and Bergelson \cite{BB84}, consist of a composition of different maps, represented by $T_k\circ T_{k-1}\circ\cdots\circ T_1$. The literature on statistical properties for such systems is already extensive. Conze and Raugi's seminal paper \cite{Conze07} explored the dynamical Borel-Cantelli lemma and the central limit theorem (CLT) for a sequence of one-dimensional piecewise expanding maps. Haydn et al \cite{Haydn17} further investigated the almost sure invariance principle (ASIP) for sequential dynamical systems and some other non-stationary systems, which implies the CLT, the law of the iterated logarithm and their functional forms. Hafouta \cite{H20} obtained the Berry-Esseen theorem for sequential dynamical systems. Additionally, the extreme value theory \cite[Section~3-4]{FFV17} and concentration inequality \cite{AR16} were also obtained for sequential dynamical systems.

Random dynamical systems, as a particular case of time-dependent systems, have also attracted a lot of attention over the past few decades. For example, Buzzi \cite{B99} obtained exponential decay of correlations for random piecewise expanding maps in one and higher dimensions. Aimino et al \cite{ANV15} established the annealed and quenched CLT for random expanding maps. Subsequently, Dragi\v{c}evi\'{c} et al \cite{DFGV18} proved a fiberwise ASIP for random piecewise expanding maps. Later, Dragi\v{c}evi\'{c} and Hafouta \cite{DH21} extended Gou\"ezel's spectral approach to obtain the vector-valued ASIP.

Notably, for the systems discussed in the references mentioned above, their transfer operators with respect to the Lebesgue measure are quasi-compact on a suitable Banach space. However, when considering the composition of Pomeau-Manneville-like maps, obtained by perturbing the slope at the indifferent fixed point 0, the transfer operators  are not quasi-compact. Noteworthy results in this situation include discussions on
the loss of memory \cite{AHNTV15,KL22}, the extreme value law \cite{FFV18}, the CLT \cite{NTV18}, the ASIP \cite{S19}, the large deviation \cite{Nicol21}, among others.
We point out that the results in \cite{AHNTV15,KL22,FFV18} are applicable to sequential dynamical systems and the results in \cite{NTV18,S19,Nicol21} are applicable to both sequential and random dynamical systems.

In the present paper, we focus on the rate of convergence with respect to the Wasserstein distance in the invariance principle for sequential dynamical systems, whose transfer operators are quasi-compact in the setting of \cite{Conze07,Haydn17}. The invariance principle (also known as the functional CLT) states that a stochastic process constructed by the sums of random variables with suitable scale converges weakly to a Brownian motion. Here, we employ the Wasserstein distance to measure the rate of weak convergence.  For $p\ge 1$, we denote by $\cal W_p(P,Q)$ the Wasserstein distance between the distributions $P$ and $Q$ on a Polish space $(\cal X,d)$ (see \cite[Definition~6.1]{V09}):
\[
\cal W_p(P,Q)= \inf \{ [\mathbb{E} {d(X,Y)}^p]^{1/p}; \hbox{law} (X)=P, \hbox{law} (Y)=Q \}.
\]
In comparison to the L\'{e}vy-Prokhorov distance, the Wasserstein distance is stronger and contains more information since it involves the metric of the underlying space. This distance finds important applications in the fields of optimal transport, geometry, partial differential equations etc; see e.g.\ Villani~\cite{V09} for details.

To the best of our knowledge, there are only few results in the literature regarding the convergence rate in the weak invariance principle (WIP) for dynamical systems. Early works on the convergence rates in the WIP for the deterministic dynamical systems go back to \cite{AM19, Liu23}. Antoniou and Melbourne \cite{AM19} established the convergence rate in the L\'evy-Prokhorov distance in the WIP for nonuniformly hyperbolic systems. Liu and Wang \cite{Liu23} obtained the Wasserstein convergence rate in
the same setting. For the non-stationary case, in the probability theory literature, Hafouta \cite{H23} obtained the convergence rate (in the L\'evy-Prokhorov distance) in the invariance principle for $\alpha$-mixing triangular arrays that is also applicable to some classes of sequential expanding systems like non-stationary subshifts. Dedecker et al \cite{DMR22} provided rates of convergence (in the
$\cal W_1$-distance and the Kolmogorov distance) in the CLT for martingale in the non-stationary setting.  Turning to the dynamical systems literature, Hella and Lepp\"anen \cite{HL20} obtained the convergence rate (in the $\cal W_1$-distance) in the CLT for time-dependent intermittent maps.

In sequential dynamical systems, the variance can grow at an arbitrarily slow rate. In most limit theorem results we reference, the variance grows linearly, or specific growth conditions are imposed on the variance. Recently, Dolgopyat and Hafouta \cite{DH24} established the Berry-Esseen theorem and the almost sure invariance principle with rates for sequential dynamical systems without assuming any growth conditions on the variance.

In this paper, without any assumptions on the growth of variance, we obtain the Wasserstein convergence rate $O(\Sigma_n^{-\frac{1}{2}+\delta})$ in the invariance principle for sequential dynamical systems, where $\Sigma_n^2$ denotes the variance and $\delta$ can be arbitrarily small. To derive the convergence rate, we employ techniques developed for stationary systems, particularly the martingale approximation method and the martingale Skorokhod embedding theorem. A key component are the moment estimates (Propositions~\ref{VV} and \ref{mom})
from \cite{DH24}, which allow us to remove the growth condition on the variance.
The convergence rate we obtain is close to the best one achieved in the i.i.d. case. Additionally, we apply our result to a class of dynamical systems, including sequential $\beta_n$-transformations, piecewise uniformly expanding maps with additive noise in one-dimensional and multidimensional case, and a general class of covering maps. We point out that the family of maps we consider consists of maps which are sufficiently close to a fixed map.

To establish the related convergence rate in the invariance principle for sequential dynamical systems, whose transfer operators are not quasi-compact,  a secondary martingale-coboundary decomposition \cite{KKM18}, similar to that in the stationary case, may be the key. However, the decomposition is currently unavailable and it is the ongoing focus of our research.

The remainder of this paper is organized as follows. In Section 2, we introduce the setting and main result of this paper. In Section 3, we recall the martingale decomposition for sequential dynamical systems and give results on moment estimates.
In Section 4, we prove the main result. In the last section, we give some applications to explain our result.

Throughout the paper, we use $1_A$ to denote the indicator function of measurable set $A$. As usual, $a_n=O(b_n)$ means that there exists a constant $C>0$ such that $|a_n|\le C |b_n|$ for all $n\ge 1$, and $\|\cdot\|_{p}$ means the $L^p$-norm. For simplicity we write $C$ to denote constants independent of $n$ and $C$ may change from line to line.  We use $\rightarrow_{w}$ to denote the weak convergence in the sense of probability measures \cite{Bill99}. We denote by $C[0,1]$ the space of all continuous functions on $[0,1]$ equipped with the supremum distance $d_C$, that is
\[
d_C(x,y):=\sup_{t\in [0,1]}|x(t)-y(t)|, \quad x,y\in C[0,1].
\]
We use $\P_X$ to denote the law/distribution of random variable $X$ and use $X=_d Y$ to mean $X, Y$ sharing the same distribution. We use the notation $\mathcal{W}_p(X,Y)$ to mean $\mathcal{W}_p(\P_X, \P_Y)$ for the sake of simplicity.
\section{Setting and main result}
In this section, we first recall an introduction to sequential dynamical systems and some basic assumptions, which were described in detail
in~\cite{Conze07,Haydn17}, and then we state our main result.
\subsection{Sequential dynamical systems}
Let $M$ be a compact subset of $\R^d$ or a torus $\T^d$ with the Lebesgue measure $m$. Consider a family $\mathcal F$ of non-invertible maps $T_\alpha:M\to M$, which are non-singular with respect to $m$ (i.e. $m(T_\alpha^{-1}E)=0$ if and only if $m(E)=0$ for all Borel measurable sets $E\subset M$).
We take a countable sequence of maps
$\{T_k\}_{k\ge 1}$ from $\mathcal F$; this sequence defines a sequential dynamical system.

We denote by $\{\cal{T}^n\}_{n\ge 0}$ the sequence of composed maps
\[
\cal{T}^n:=T_n\circ T_{n-1}\circ\cdots\circ T_1 \quad\text{for~} n\ge 1, \text{~and~} \cal T^0:=Id.
\]
The transfer operator $P_\alpha$ corresponding to $T_\alpha$ is defined by
\[
\int_{M}P_\alpha f\cdot g\rmd m=\int_{M} f\cdot g\circ T_\alpha\rmd m \quad\hbox{for~all~} f\in L^1(m), g\in L^\infty(m).
\]
Similar to $\cal{T}^n$, we can define the composition of operators as
\[
\cal{P}^n:=P_n\circ P_{n-1}\circ\cdots\circ P_1 \quad\text{for~} n\ge 1, \text{~and~} \cal P^0:=Id.
\]
Then it is easy to check that
\begin{align}\label{trans}
\int_{M} \cal{P}^n f\cdot g\rmd m=\int_{M} f\cdot g\circ \cal{T}^n\rmd m \quad\hbox{for~all~} f\in L^1(m), g\in L^\infty(m).
\end{align}

For a fixed sequence $\{\cal T^n\}_{n\ge 0}$, we set $\cal B_n:=(\cal T^n)^{-1}\cal B$, the $\sigma$-algebra associated with $n$-fold pull back of the Borel $\sigma$-algebra $\cal B$. Since the transformations $T_n$ are non-invertible,
we obtain a decreasing sequence of $\sigma$-algebras $\{\cal B_n\}_{n\ge 0}$, i.e. $\cal B_n\subset \cal B_m$ for $n\ge m\ge 0$. It was described
in~\cite{Conze07} that for $f\in L^{\infty}(m)$, the quotients $|\cal{P}^n f/\cal{P}^n 1|$ are bounded by $\|f\|_\infty$ on the set
$\{\cal{P}^n 1>0\}$ and we have $\cal{P}^n f(x)=0$ on $\{\cal{P}^n 1=0\}$. Then we can define $|\cal{P}^n f/\cal{P}^n 1|=0$ on $\{\cal{P}^n 1=0\}$.
Therefore, we have
\begin{align}
\E(f|\cal B_k)=\big(\frac{\cal{P}^k f}{\cal{P}^k 1}\big)\circ \cal T^k,
\end{align}
and,
\begin{align}\label{cexp}
\E(f\circ \cal T^l|\cal B_k)=\big(\frac{P_k\cdots P_{l+1}(f\cal{P}^l 1)}{\cal{P}^k 1}\big)\circ \cal T^k,\quad 0\le l\le k\le n.
\end{align}
Here, the expectation is taken with respect to the Lebesgue measure $m$.
\subsection{Assumptions}
Let $\cal V\subset L^1(m) (1\in \cal V)$ be a Banach space of  functions from $M$ to $\R$ with norm $\|\cdot\|_\alpha$, such that $\|v\|_\infty\le C\|v\|_\alpha$ for some constant $C>0$ independent of $v$.
For example, we can let $\cal V$ be the Banach space of bounded variation functions on a compact interval of $\R$ with the norm $\|\cdot\|_{BV}$ given by the sum of the $L^1$ norm and the total variation $|\cdot|_{bv}$, or we can take $\cal V$ to be the space of $\alpha$-H\"older functions on a compact set of $\R^d$ with the norm $\|\cdot\|_\alpha=\|\cdot\|_\infty+|\cdot|_\alpha$, where $|\cdot|_\alpha$
denotes the H\"older semi-norm.

Following the setting described in~\cite{Conze07} and \cite{Haydn17}, we now recall the required properties (DEC) and (MIN). Moreover, we add a property (SUP), which is implied in~\cite{Conze07}.

{\bf Property (DEC).} Given a family $\mathcal F$ of non-invertible non-singular maps defined on $M$, there exist constants $C>0$, $\gamma\in(0,1)$ such that for any $n\ge 1$, any sequence of transfer operators $P_1, P_2,\ldots, P_n$ corresponding to maps chosen from $\mathcal F$ and any $v\in \cal V$ with zero (Lebesgue) mean, we have
\[
\|P_n\circ P_{n-1}\circ\cdots\circ P_1v\|_\alpha\le C\gamma^n\|v\|_\alpha.
\]

{\bf Property (MIN).} There exists $\delta>0$ such that for any sequence $P_1, P_2,\ldots, P_n$  as defined above, we have the uniform lower bound
\[
\inf_{x\in M}P_n\circ P_{n-1}\circ\cdots\circ P_11(x)\ge \delta, \quad\forall n\ge 1.
\]

{\bf Property (SUP).} For any sequence $P_1, P_2,\ldots, P_n$ as defined in (DEC), we have
\[
\sup_n\|P_n\circ P_{n-1}\circ\cdots\circ P_11\|_\infty< \infty.
\]

\subsection{Main result}
Let $v_n:M\to\R$ be a family of functions in $\cal V$ such that $\sup_n\|v_n\|_\alpha<\infty$.
Denote $S_n\bar v:=\sum_{i=0}^{n-1}\bar v_i\circ \cal T^i$, $\Sigma_n^2:=\E(\sum_{i=0}^{n-1}\bar v_i\circ \cal T^i)^2$, where
$\bar v_i:=v_i-\int_M v_i\circ \cal T^i\rmd m$. For every $t\in[0,1]$, set
\[
N_n(t):=\min\{1\le k\le n: t\Sigma_n^2\le\Sigma_{k}^2\}.
\]
Consider the following continuous processes $W_{n}(t)\in C[0,1]$ defined by
\begin{equation}\label{wnt}
W_{n}(t):=\frac{1}{\Sigma_n}\bigg[\sum_{i=0}^{N_n(t)-1}\bar v_i\circ \cal T^i+\frac{t\Sigma_n^2-\Sigma_{N_n(t)-1}^2}{\Sigma_{N_n(t)}^2-\Sigma_{N_n(t)-1}^2}\bar v_{N_n(t)}\circ \cal T^{N_n(t)}\bigg],\quad t\in[0,1].
\end{equation}

When the sequence $\{\cal T^n\}$ satisfies (DEC) and (MIN), and the variance $\Sigma_n^2$ satisfies an additional growth rate condition, i.e. $\Sigma_n\ge n^{\frac{1}{4}+\delta}$ for some $0<\delta<\frac{1}{4}$, Haydn et al \cite{Haydn17} obtained that the almost sure invariance principle (ASIP) holds. Recently, Dolgopyat and Hafouta \cite{DH24} improved the result by removing the assumption on the growth of variance and, in a more general setting, obtained the ASIP.
Namely, for any $\delta>0$, there is, enlarging the probability space if necessary, a sequence of independent centered Gaussian variables $\{Z_k\}$ such that
\[
\sup_{1\le k\le n}\big|\sum_{i=1}^{k}\bar v_i\circ \cal T^i-\sum_{i=1}^{k}Z_i\big|=o(\Sigma_n^{1/2+\delta}) \quad m-a.s.
\]

We can deduce from the ASIP that the weak invariance principle holds, i.e. $W_n\to_w B$ in $C[0,1]$, where $B$ is a standard Brownian motion. Now, we introduce our main result on the Wasserstein convergence rate in the invariance principle.
\begin{thm}\label{thnon}
Assume that $\{\cal T^n\}$ satisfies (DEC), (MIN) and (SUP). Let $\{v_n\}$ be a sequence of functions in $\cal V$ with $\sup_n\|v_n\|_\alpha<\infty$. Then for any $\delta>0$, there exists a constant $C>0$ such that $\mathcal{W}_{p}(W_{n},B)\leq C \Sigma_n^{-\frac{1}{2}+\delta}$ for $n\ge 1$ and $p\ge 2$, where $B$ is a standard Brownian motion.
\end{thm}

\begin{rem}
Our result implies a convergence rate $\pi(W_n, B)=O(\Sigma_n^{-\frac{1}{2}+\delta})$ with respect to the L\'{e}vy-Prokhorov distance. Indeed, for any two given probability measures $\mu$ and $\nu$, we have $\pi(\mu,\nu)\le \mathcal{W}_{p}(\mu,\nu)^{\frac{p}{p+1}}$ for
$p\ge 1$.
\end{rem}

\begin{rem}
Our result can be applied to random dynamical systems in the setting of \cite{DFGV18, DH21}. In random dynamical systems, the variance typically grows linearly. Nevertheless, we still should consider the self-normalized Birkhoff sums. Namely, the continuous process under consideration should be defined in the same way as in the sequential case. To our understanding, in the non-self-normalized case, it is a tricky problem to get the convergence rate of quenched variance to the annealed variance, because we know nothing about the regularity of the observable of the base map.
\end{rem}

\begin{rem}
Note that our method does not work for the estimate of $\mathcal{W}_1(W_n, B)$. But we know that $\mathcal{W}_{q}(W_{n},B)$ $\le \mathcal{W}_{p}(W_{n},B)$ for $q\le p$, so $\mathcal{W}_{1}(W_{n},B)$ can be controlled by $\mathcal{W}_{q}(W_{n},B)$ for $q>1$. It seems an interesting question to estimate the convergence rate for $\mathcal{W}_1(W_n, B)$ directly, which probably produces a better rate.
\end{rem}

\section{Moment estimates}
In the following, we assume that $\{\cal T^n\}$ satisfies the conditions (DEC), (MIN) and (SUP).
As in \cite{Conze07}, we define the operator $Q_n$ by $Q_nv:=\frac{P_n(v\cal P^{n-1}1)}{\cal P^{n}1}$. Set $h_0:=0$ and for $n\ge 1$,
\begin{align*}
h_n&:=Q_n\bar v_{n-1}+Q_n\circ Q_{n-1}\bar v_{n-2}+\cdots+Q_n\circ Q_{n-1}\circ \cdots \circ Q_1\bar v_{0}\\
&=\frac{1}{\cal P^{n}1}\big[P_n(\bar v_{n-1}\cal P^{n-1}1)+P_n\circ P_{n-1}(\bar v_{n-2}\cal P^{n-2}1)+\cdots+P_n\circ P_{n-1}\circ \cdots \circ P_{1}(\bar v_{0}\cal P^{0}1)\big].
\end{align*}
Since $\{\bar v_{n-k}\cal P^{n-k}1\}_{1\le k\le n}$ belongs to $\cal V$, by the properties (DEC) and (MIN), $\|h_n\|_\alpha$ is uniformly bounded. In particular, $h_n\in L^{\infty}(m)$.

Define $\psi_n=\bar v_{n}+h_n-h_{n+1}\circ T_{n+1}$. Then $\|\psi_n\|_\infty\le \|\bar v_{n}\|_\infty+2\|h_n\|_\infty<\infty$.
 It follows from~\cite{Conze07} that $\{\psi_{n}\circ \mathcal T^n\}_{n\ge0}$ is a sequence of reverse martingale differences for the filtration $\{\cal B_n\}_{n\ge 0}$, and we have
\[
\sum_{i=0}^{n-1}\bar v_{i}\circ \mathcal T^i=\sum_{i=0}^{n-1}\psi_{i}\circ \mathcal T^i+h_{n}\circ \mathcal T^{n}.
\]

\begin{prop}\label{SD}
$\Sigma_{n}=\sigma_{n}+O(1)$, where $\sigma_n^2=\mathbb \sum_{i=0}^{n-1}\E(\psi_i^2\circ \mathcal T^i)$.
\end{prop}
\begin{proof}
Since
\begin{align*}
&|\Sigma_{n}-\sigma_{n}|=\Big|\big\|\sum_{i=0}^{n-1}\bar v_{i}\circ \mathcal T^i\big\|_2-\big\|\sum_{i=0}^{n-1}\psi_{i}\circ \mathcal T^i\big\|_2\Big|\\
&\le \Big\|\sum_{i=0}^{n-1}\bar v_{i}\circ \mathcal T^i-\sum_{i=0}^{n-1}\psi_{i}\circ \mathcal T^i\Big\|_2=\|h_{n}\circ \mathcal T^{n}\|_2<\infty,
\end{align*}
the result follows.
\end{proof}

Next, we introduce the moment estimates for the maxima of partial sums. These are modifications of \cite[Proposition 3.3]{DH24} and \cite[Proposition 6.6]{DH24}. We denote $\Psi_i=\psi_{i}^2$ and $S_n\Psi=\sum_{i=0}^{n-1}\Psi_{i}\circ \mathcal T^i$. Recall that
$S_n\bar v=\sum_{i=0}^{n-1}\bar v_{i}\circ \mathcal T^i$.
\begin{prop}\label{VV}
There exists a constant $C>0$ (independent of $n$) such that for all $n\ge 1$,
\[
Var(S_n\Psi)\le C(1+Var(S_n\bar v)).
\]
\end{prop}
\begin{proof}
Denote $g_i=\Psi_i-\int_M \Psi_i\circ \cal T^i\rmd m$ and $S_ng=\sum_{i=0}^{n-1}g_{i}\circ \mathcal T^i$.
Note that $\sup_n\|g_n\|_\alpha<\infty$ and $\sup_n\|\mathcal P^n1\|_\infty<\infty$.
It follows from \eqref{trans} and the property (DEC) that
\begin{align*}
Var(S_n\Psi)&=\E[(S_ng)^2]\le 2\sum_{0\le l<n}\sum_{0\le k\le l} |\int (g_{l}\circ \mathcal T^l)(g_{k}\circ \mathcal T^k)\rmd m|\\
&=2\sum_{0\le l<n}\sum_{0\le k\le l} |\int (g_{l}\circ T_l\circ T_{l-1}\circ \cdots \circ T_{k+1}\cdot g_{k})\circ \mathcal T^k\rmd m|\\
&=2\sum_{0\le l<n}\sum_{0\le k\le l} |\int (g_{k}\mathcal P^k1)g_{l}\circ T_l\circ T_{l-1}\circ \cdots \circ T_{k+1}\rmd m|\\
&\le 2\sup_n\|\mathcal P^n1\|_\infty\sum_{0\le l<n}\sum_{0\le k\le l} |\int g_{k}\cdot g_{l}\circ T_l\circ T_{l-1}\circ \cdots \circ T_{k+1}\rmd m|\\
&= C\sum_{0\le l<n}\sum_{0\le k\le l} |\int P_l\circ P_{l-1}\circ \cdots \circ P_{k+1}(g_{k})\cdot g_{l}\rmd m|\\
&\le C\sum_{0\le l<n}\sum_{0\le k\le l}\int |g_{l}|\rmd m\cdot \|P_l\circ P_{l-1}\circ \cdots \circ P_{k+1}(g_{k})\|_\alpha\\
&\le C\sum_{0\le l<n}\int |g_{l}|\rmd m \Big(\sum_{0\le k\le l} \gamma^{l-k}\|g_{k}\|_\alpha\Big)\\
&\le C\sum_{0\le l<n}\int |g_{l}|\rmd m\le C\sum_{0\le l<n}\int \Psi_l\rmd m\\
&\le C\frac{1}{\delta}\sum_{0\le l<n}\int \Psi_l\mathcal P^l1\rmd m\\
&=  C\frac{1}{\delta}\sum_{0\le l<n}\int \Psi_l\circ \cal T^l\rmd m.
\end{align*}
Since $\sum_{0\le l<n}\int \Psi_l\circ \cal T^l\rmd m=Var(S_n\psi)$ and $Var(S_n\psi)\le C(1+Var(S_n\bar v))$, the result follows.
\end{proof}

\begin{prop}\label{mom}
For every $p\ge 2$, there exists a constant $C>0$ (independent of $n$) such that for all $n\ge 1$,
\[
\Big\|\max_{1\le k\leq n}|\sum_{i=0}^{k-1}\bar v_{i}\circ \mathcal T^i|\Big\|_{p}\leq C\Big(1+\Big\|\sum_{i=0}^{n-1}\bar v_{i}\circ \mathcal T^i\Big\|_{2}\Big).
\]
\end{prop}
\begin{proof}
It is enough to show the result for $p=2^m$ for all $m\ge 1$. We use induction on $m$. When $m=1$, since $\{\psi_{n-i}\circ \mathcal T^{n-i}\}_{1\le i\le n}$ is a sequence of martingale differences, by Doob's martingale inequality and Proposition~\ref{SD},
\begin{align*}
&\Big\|\max_{1\le k\leq n}|\sum_{i=1}^{k}\bar v_{n-i}\circ \mathcal T^{n-i}|\Big\|_{2}\\
\le &\Big\|\max_{1\le k\leq n}|\sum_{i=1}^{k}\psi_{n-i}\circ \mathcal T^{n-i}|\Big\|_{2}+\max_{1\le k\leq n}\|h_k\|_\alpha\\
\le &4\Big\|\sum_{i=1}^{n}\psi_{n-i}\circ \mathcal T^{n-i}\Big\|_{2}+\max_{1\le k\leq n}\|h_k\|_\alpha\\
\le &C(1+\|S_n\bar v\|_2).
\end{align*}
The result for $m=1$ holds. We assume that the statement is true for some $m>1$, that is
\begin{align}\label{2m}
\Big\|\max_{1\le k\leq n}|\sum_{i=1}^{k}\bar v_{n-i}\circ \mathcal T^{n-i}|\Big\|_{2^m}\leq C(1+\|S_n\bar v\|_2).
\end{align}
We aim to estimate $\big\|\max_{1\le k\leq n}|\sum_{i=1}^{k}\bar v_{n-i}\circ \mathcal T^{n-i}|\big\|_{2^{m+1}}$.
Similar with the argument for $m=1$, we have
\[
\Big\|\max_{1\le k\leq n}|\sum_{i=1}^{k}\bar v_{n-i}\circ \mathcal T^{n-i}|\Big\|_{2^{m+1}}\le C_m\Big\|\sum_{i=1}^{n}\psi_{n-i}\circ \mathcal T^{n-i}\Big\|_{2^{m+1}}+\max_{1\le k\leq n}\|h_k\|_\alpha.
\]
It suffices to prove that
\[
\Big\|\sum_{i=1}^{n}\psi_{n-i}\circ \mathcal T^{n-i}\Big\|_{2^{m+1}}\le C(1+\|S_n\bar v\|_2).
\]
By Burkholder's inequality,
\begin{align}\label{m+1}
\Big\|\sum_{i=1}^{n}\psi_{n-i}\circ \mathcal T^{n-i}\Big\|_{2^{m+1}}\le C'_m\Big\|\sum_{i=1}^{n}\psi_{n-i}^2\circ \mathcal T^{n-i}\Big\|_{2^{m}}^{1/2}.
\end{align}
Applying \eqref{2m} to $g_i=\psi_i^2-\int_M \psi_i^2\circ \cal T^i\rmd m$, we have
\[
\big\|\max_{1\le k\leq n}|\sum_{i=1}^{k} g_{n-i}\circ \mathcal T^{n-i}|\big\|_{2^m}\leq C(1+\|S_ng\|_2).
\]
We can estimate that
\begin{align*}
\Big\|\sum_{i=1}^{n}\psi_{n-i}^2\circ \mathcal T^{n-i}\Big\|_{2^{m}}&\le \big\|S_ng\big\|_{2^m}+\E\Big(\sum_{i=0}^{n-1}\psi_{i}^2\circ \mathcal T^i\Big)\\
&\le C(1+\|S_ng\|_2)+\E\Big(\sum_{i=0}^{n-1}\psi_{i}^2\circ \mathcal T^i\Big)\\
&\le C(1+C(1+Var(S_n\bar v)))+\E\Big(\sum_{i=0}^{n-1}\psi_{i}^2\circ \mathcal T^i\Big),
\end{align*}
where the last inequality is due to Proposition~\ref{VV}.
Note that
$\E\Big(\sum_{i=0}^{n-1}\psi_{i}^2\circ \mathcal T^i\Big)=Var(S_n\psi)$ and $Var(S_n\psi)\le C(1+Var(S_n\bar v))$. Combining with \eqref{m+1}, we have
\[
\Big\|\sum_{i=1}^{n}\psi_{n-i}\circ \mathcal T^{n-i}\Big\|_{2^{m+1}}\le C(1+Var(S_n\bar v)^{1/2})=C(1+\|S_n\bar v\|_2).
\]
Writing $\sum_{i=0}^{k-1}\bar v_{i}\circ \mathcal T^i=\sum_{i=1}^{n}\bar v_{n-i}\circ \mathcal T^{n-i}-\sum_{j=1}^{n-k}\bar v_{n-j}\circ \mathcal T^{n-j}$, we can obtain the result.
\end{proof}

\begin{rem}\label{mar}
By the argument of the proof of Proposition~\ref{mom}, we also obtain the result for $\psi_{i}$. Namely, for every $p\ge 2$,
\[\Big\|\max_{k\leq n-1}\big|\sum_{i=0}^{k-1}\psi_{i}\circ \mathcal T^{i}\big|\Big\|_{p}\leq C(1+\|S_n\psi\|_2).\]
\end{rem}

\begin{cor}\label{var}
$\Sigma_{n}^{2}=\sigma_{n}^{2}+O(\sigma_n)$, where $\sigma_n^2=\E(\sum_{i=0}^{n-1}\psi_i\circ \cal T^i)^2$.
\end{cor}
\begin{proof}
We can write
\begin{align*}
&\Sigma_{n}^{2}-\sigma_{n}^{2}
=\int \Big(\sum_{i=0}^{n-1}\bar v_{i}\circ \mathcal T^i\Big)^2{\rm d}m-\int \Big(\sum_{i=0}^{n-1}\psi_{i}\circ \mathcal T^i\Big)^2{\rm d}m\\
=&\int \Big(\sum_{i=0}^{n-1}\bar v_{i}\circ \mathcal T^i-\sum_{i=0}^{n-1}\psi_{i}\circ \mathcal T^i\Big)\Big(\sum_{i=0}^{n-1}\bar v_{i}\circ \mathcal T^i+\sum_{i=0}^{n-1}\psi_{i}\circ \mathcal T^i\Big){\rm d}m\\
=&\int h_{n}\circ \mathcal T^{n}\Big(2\sum_{i=0}^{n-1}\psi_{i}\circ \mathcal T^i+h_{n}\circ \mathcal T^{n}\Big){\rm d}m\\
=&\int h_{n}^2\circ \mathcal T^{n}{\rm d}m+2\int \Big(\sum_{i=0}^{n-1}\psi_{i}\circ \mathcal T^i\Big)h_{n}\circ \mathcal T^{n}{\rm d}m\\
\le &\big\|h_n\big\|_\infty^2+2\Big\|\sum_{i=0}^{n-1}\psi_{i}\circ \mathcal T^i\Big\|_2\big\|h_{n}\circ \mathcal T^{n}\big\|_2.	
\end{align*}
Then the result follows from Remark~\ref{mar}.
\end{proof}

\section{Proof of Theorem \ref{thnon}}
\setcounter{equation}{0}
Recall that $\sigma_n^2=\E(\sum_{i=0}^{n-1}\psi_i\circ \cal T^i)^2=\sum_{i=0}^{n-1}\E(\psi_i^2\circ \cal T^i)$. For every $t\in[0,1]$, set
\[
r_n(t):=\min\{1\le k\le n: t\sigma_n^2\le\sigma_{k}^2\}.
\]
Similar to $W_n$, we define the following continuous processes $M_{n}(t)\in C[0,1]$ by
\begin{equation}\label{exp}
M_{n}(t):=\frac{1}{\sigma_n}\bigg[\sum_{i=0}^{r_n(t)-1}\psi_i\circ \cal T^i+\frac{t\sigma_n^2-\sigma_{r_n(t)-1}^2}{\sigma_{r_n(t)}^2-\sigma_{r_n(t)-1}^2}\psi_{r_n(t)}\circ \cal T^{r_n(t)}\bigg],\quad t\in[0,1].
\end{equation}

{\bf Step 1. Estimation of the convergence rate between $W_n$ and $M_n$.}
\begin{lem}\label{wmn}
Let $p\ge 2$. Then for any $\delta>0$, there exists a constant $C>0$ such that for all $n\ge1$,
\[\Big\|\sup_{t\in[0,1]}|W_{n}(t)-M_n(t)|\Big\|_{p}\leq C\Sigma_n^{-\frac{1}{2}+\delta}.\]
\end{lem}
\begin{proof}
By Corollary~\ref{var}, there exists a constant $B>0$ such that
\begin{align}\label{C}
\Sigma_n^2\le \sigma_n^2+B\Sigma_n \quad\text{~~or~~}\quad \sigma_n^2\le \Sigma_n^2+B\Sigma_n.
\end{align}
Similar to the construction of the intervals in \cite[Section~4.2]{H23},
we take $b_1$ to be the first value in $\mathbb N$ such that
\[
2B\Sigma_n\le \E\big(\sum_{i=0}^{b_1-1}\bar v_i\circ \cal T^i\big)^2\le \E\big(\max_{1\le m\le b_1}|\sum_{i=0}^{m-1}\bar v_i\circ \cal T^i|\big)^2\le 4B\Sigma_n.
\]
Let $b_2>b_1$ be the smallest value in $\mathbb N$ such that
\[
2B\Sigma_n\le \E\big(\sum_{i=b_1}^{b_2-1}\bar v_i\circ \cal T^i\big)^2\le \E\big(\max_{b_1+1\le m\le b_2}|\sum_{i=b_1}^{m-1}\bar v_i\circ \cal T^i|\big)^2\le 4B\Sigma_n.
\]
Continuing this way, we decompose $\{0,1,\ldots, n-1\}$ into a disjoint union of intervals $I_1,\ldots, I_{Q_n}$ in $\mathbb N$ such that:\\
%
(i) $I_j$ is to the left of $I_{j+1}$, denoted by $I_j=\{a_j,\ldots, b_j\}$, where $a_1=0$ and $a_j=b_{j-1}+1$ for $j\ge 2$;\\
(ii) for each $1\le j\le Q_n$,
\begin{align}\label{decom}
2B\Sigma_n\le \E(S_{I_j}\bar v)^2\le \E\big(\max_{m\in I_j}|\sum_{i=a_j}^{m}\bar v_i\circ \cal T^i|\big)^2\le 4B\Sigma_n,
\end{align}
where $S_I\bar v=\sum_{j\in I}\bar v_j\circ \cal T^j$ for each interval $I$ and the constant $B$ is from \eqref{C}. We point out that in the above construction we may need to absorb the last interval in the penultimate one, but this simply requires replacing $2B\Sigma_n$ with $4B\Sigma_n$, which makes no difference in the following arguments.

We know that for each $t\in [0,1]$, there exists $1\le J\le Q_n$ such that
$N_n(t)\in I_{J}$.
By the condition (ii), we know that $r_n(t)$ is in the same interval $I_J$, or in the adjacent intervals $I_{J-1}$ or $I_{J+1}$.
Indeed, if $r_n(t)\in I_{J+2}$, then by \eqref{C}, we have
\[
\E(S_{I_{J+1}}\bar v)^2\le t\sigma_n^2-t\Sigma_n^2\le B\Sigma_n,
\]
which is a contradiction with (ii). Similarly, if $r_n(t)\in I_{J-2}$,
\[
\E(S_{I_{J-1}}\bar v)^2\le t\Sigma_n^2-t\sigma_n^2\le B\Sigma_n,
\]
which is also a contradiction with (ii).

Since $\Sigma_n=\sigma_n+O(1)$, we have $\|S_{I_j}\bar v\|_2\le \|S_{I_j}\psi\|_2+O(1)$. Then by certain calculations,
we have $C_1\le Q_n/\Sigma_n\le C_2$ for some constants $C_1, C_2$ depending only on $B$.

Consider $\overline W_{n}(t):=\frac{1}{\Sigma_n}\sum_{i=0}^{b_{J(t)}-1}\bar v_i\circ \cal T^i$. Then
\[
\sup_{t\in[0,1]}|W_n(t)-\overline W_{n}(t)|\le \frac{1}{\Sigma_n}\max_{1\le j\le Q_n}|Z_j|+\frac{1}{\Sigma_n}\max_{1\le j\le n}|\bar v_j\circ \cal T^j|,
\]
where $Z_j=\max_{m\in I_j}|S_{m}\bar v-S_{a_j}\bar v|$. By Proposition~\ref{mom} and \eqref{decom}, for all $p\ge 2$,
\[\|Z_j\|_p\le \big\|\max_{m\in I_j}|S_{m}\bar v-S_{a_j}\bar v|\big\|_p\le C'(\Sigma_n)^{1/2}.\]
Then for any $\kappa>1$, by Proposition~\ref{mpe},
\begin{align}\label{est}
&\nonumber\Big\|\sup_{t\in[0,1]}|W_{n}(t)-\overline W_n(t)|\Big\|_{p}\\
&\nonumber\le \frac{1}{\Sigma_n}\Big\|\max_{1\le j\le Q_n}|Z_j| \Big\|_{p}+\frac{1}{\Sigma_n}\max_{1\le j\le n}\|\bar v_j\|_\infty\\
&\nonumber\le \frac{1}{\Sigma_n}\Big\|\max_{1\le j\le Q_n}|Z_j| \Big\|_{\kappa p}+\frac{1}{\Sigma_n}\max_{1\le j\le n}\|\bar v_j\|_\infty\\
&\nonumber\le \frac{1}{\Sigma_n}(Q_n)^{\frac{1}{\kappa p}}\max_{1\le j\le Q_n}\big\|Z_j\big\|_{\kappa p}+\frac{1}{\Sigma_n}\max_{1\le j\le n}\|\bar v_j\|_\infty\\
&\le C\Sigma_n^{-\frac{1}{2}+\frac{1}{\kappa p}}+C\Sigma_n^{-1}\le  C\Sigma_n^{-\frac{1}{2}+\delta}
\end{align}
by choosing $\kappa$ large enough.

Similarly, consider $\overline M_{n}(t):=\frac{1}{\sigma_n}\sum_{i=0}^{b_{J'(t)}-1}\psi_i\circ \cal T^i$,
where $J'\in \{J-1, J, J+1\}$ is such that $r_n(t)\in I_{J'}$. Then
\[
\Big\|\sup_{t\in[0,1]}|M_{n}(t)-\overline M_n(t)|\Big\|_{p}\le C\sigma_n^{-\frac{1}{2}+\frac{1}{\kappa p}}\le  C\sigma_n^{-\frac{1}{2}+\delta}
\]
by choosing $\kappa$ large enough.

Finally, we aim to estimate $\big\|\sup_{t\in[0,1]}|\overline W_{n}(t)-\overline M_n(t)|\big\|_{p}$. When $J'=J$,
\begin{align*}
&\Big\|\sup_{t\in[0,1]}|\overline W_{n}(t)-\overline M_n(t)|\Big\|_{p}\\
\le &\big|\frac{1}{\Sigma_n}-\frac{1}{\sigma_n}\big|\Big\|\sup_{t\in[0,1]}\big|\sum_{i=0}^{b_{J(t)}-1}\bar v_i\circ \cal T^i\big|\Big\|_{p}+\frac{1}{\sigma_n}\Big\|\sup_{t\in[0,1]}\big|\sum_{i=0}^{b_{J(t)}-1}\bar v_i\circ \cal T^i-\sum_{i=0}^{b_{J(t)}-1}
\psi_i\circ \cal T^i\big|\Big\|_{p}
\end{align*}
For the first term, by Propositions~\ref{SD} and ~\ref{mom} , we have for $n\ge1$,
\begin{align}\label{Sig}
&\nonumber\big|\frac{1}{\Sigma_n}-\frac{1}{\sigma_n}\big|\Big\|\sup_{t\in[0,1]}\big|\sum_{i=0}^{b_{J(t)}-1}\bar v_i\circ \cal T^i\big|\Big\|_{p}\\
= &\nonumber\big|\frac{\Sigma_n-\sigma_n}{\Sigma_n\cdot \sigma_n}\big|\Big\|\max_{1\le k\le n}\big|\sum_{i=0}^{k-1}\bar v_i\circ \cal T^i\big|\Big\|_{p}\\
\le &C\frac{1}{\Sigma_n^2}\cdot \Sigma_n= C\Sigma_n^{-1}.
\end{align}
Since $\psi_n=\bar v_{n}+h_n-h_{n+1}\circ T_{n+1}$ and $\psi_n$, $h_n$ are uniformly bounded, we can estimate the second term that
\begin{align*}
\frac{1}{\sigma_n}\Big\|\sup_{t\in[0,1]}\big|\sum_{i=0}^{b_{J(t)}-1}\bar v_i\circ \cal T^i-\sum_{i=0}^{b_{J(t)}-1}
\psi_i\circ \cal T^i\big|\Big\|_{p}
\le \frac{1}{\sigma_n}\max_{0\le j\le n-1}\|h_i\circ \cal T^i\|_\infty
\le C\sigma_n^{-1}.
\end{align*}
When $J'=J-1$, for any $\delta>0$, we have
\begin{align*}
&\Big\|\sup_{t\in[0,1]}|\overline W_{n}(t)-\overline M_n(t)|\Big\|_{p}\\
\le &\Big\|\sup_{t\in[0,1]}\big|\frac{1}{\Sigma_n}\sum_{i=0}^{b_{J(t)}-1}\bar v_i\circ \cal T^i-\frac{1}{\sigma_n}\sum_{i=0}^{b_{J(t)}-1}\bar v_i\circ \cal T^i\big|\Big\|_{p}+\frac{1}{\sigma_n}\Big\|\sup_{t\in[0,1]}\big|\sum_{i=0}^{b_{J(t)}-1}\bar v_i\circ \cal T^i-\sum_{i=0}^{b_{J(t)-1}-1}
\psi_i\circ \cal T^i\big|\Big\|_{p}\\
\le &C\Sigma_n^{-1}+\frac{1}{\sigma_n}\Big\|\max_{1\le j\le Q_n}|Z_j| \Big\|_{p}\le C\Sigma_n^{-\frac{1}{2}+\delta}.
\end{align*}
Here the first term is same as \eqref{Sig} and the estimate for the second term is similar to \eqref{est}. The argument for $J'=J+1$ is same; we omit it.
Based on the above estimates, for any $\delta>0$, we have $\big\|\sup_{t\in[0,1]}|W_{n}(t)-M_n(t)|\big\|_{p}\leq C\Sigma_n^{-\frac{1}{2}+\delta}$ for all $n\ge1$.
\end{proof}

Define
\[
\xi_{n,j}:=\frac{1}{\sigma_n}\psi_{n-j}\circ \mathcal T^{n-j},\ \mathcal{G}_{n,j}:=\cal T^{-(n-j)}\mathcal{B}, \quad\hbox{for } 1\le j\le n.
\]
Then $\{\xi_{n,j}, \mathcal{G}_{n,j}; 1\le j\le n\}$ is a martingale difference array.

For $1\le l\le n$, define the quadratic variation
\[
V_{n,l}:=\sum_{j=1}^{l}\E(\xi_{n,j}^2|\mathcal{G}_{n,j-1}).
\]
For the convenience, we set $V_{n,0}=0$.

Define the following stochastic processes $X_n$ with sample paths in $C[0,1]$ by
\begin{eqnarray}\label{xn}
X_{n}(t):=\sum_{j=1}^{k}\xi_{n,j}+\frac{tV_{n,n}-V_{n,k}}{V_{n,k+1}-V_{n,k}}\xi_{n,k+1}, \qquad \textrm{if}~~~ V_{n,k}\leq tV_{n,n}<V_{n,k+1}.
\end{eqnarray}

\medskip
{\bf Step 2. Estimation of the Wasserstein convergence rate between $X_n$ and $B$.}
\begin{prop}\label{vn1}
Let $p\ge 2$. Then there exists a constant $C>0$ such that for all $n\ge1$,
\[\big\|V_{n,n}-1\big\|_{p}\le C\sigma_n^{-1}.\]
\end{prop}
\begin{proof}
For $1\le j\le n$, we denote $\alpha_{j}^{2}=\sum_{i=1}^{j}\int\psi_{n-i}^2\circ \mathcal T^{n-i}\rmd m$. Then $\alpha_{n}^{2}=\sigma_{n}^{2}$ and
\begin{align*}
\big\|V_{n,n}-1\big\|_{p}=\big\|V_{n,n}-\frac{\alpha_{n}^{2}}{\sigma_{n}^{2}}\big\|_{p}.
\end{align*}
To deal with it, we first recall the notation that $g_{n-i}=\psi_{n-i}^2-\E(\psi_{n-i}^2\circ \cal T^{n-i})$ for $1\le i\le n$. Then we can write
\begin{align}\label{vnn1}
\big\|V_{n,n}-1\big\|_{p}&=\frac{1}{\sigma_n^2}\Big\|\sum_{i=1}^{n}\E(\psi_{n-i}^2\circ \mathcal T^{n-i}|\mathcal{G}_{n,i-1})-\sum_{i=1}^{n}\E(\psi_{n-i}^2\circ \mathcal T^{n-i})\Big\|_{p}\\\nonumber
&=\frac{1}{\sigma_n^2}\Big\|\sum_{i=1}^{n}\E\big(\psi_{n-i}^2\circ \mathcal T^{n-i}-\E(\psi_{n-i}^2\circ \mathcal T^{n-i})\big|\mathcal{G}_{n,i-1}\big)\Big\|_{p}\\\nonumber
&=\frac{1}{\sigma_n^2}\Big\|\sum_{i=1}^{n}\E\big(g_{n-i}\circ \mathcal T^{n-i}\big|\mathcal{G}_{n,i-1}\big)\Big\|_{p}\\\nonumber
&=\frac{1}{\sigma_n^2}\Big\|\sum_{i=1}^{n}\frac{P_{n-i+1}(g_{n-i}\cdot\cal{P}^{n-i} 1)}{\cal{P}^{n-i+1} 1}\circ \cal T^{n-i+1}\Big\|_{p},
\end{align}
where the last equation is due to \eqref{cexp}.

We claim that $\frac{P_{n-i+1}(g_{n-i}\cdot\cal{P}^{n-i} 1)}{\cal{P}^{n-i+1} 1}\in \cal V$ and
$\E\big(\frac{P_{n-i+1}(g_{n-i}\cdot\cal{P}^{n-i} 1)}{\cal{P}^{n-i+1} 1}\circ \cal T^{n-i+1}\big)=0$.
Imitating the proof of Proposition~\ref{VV}, we obtain that
\[
\Big\|\sum_{i=1}^{n}\frac{P_{n-i+1}(g_{n-i}\cdot\cal{P}^{n-i} 1)}{\cal{P}^{n-i+1} 1}\circ \cal T^{n-i+1}\Big\|_{2}\le C\Big\|\sum_{i=1}^{n}\psi_{n-i}\circ \mathcal T^{n-i}\Big\|_2.
\]
Then by Proposition~\ref{mom}, for $n\ge1$,
\begin{align*}
&\big\|V_{n,n}-1\big\|_{p}\le \frac{C}{\sigma_n^2}\Big(1+\Big\|\sum_{i=1}^{n}\frac{P_{n-i+1}(g_{n-i}\cdot\cal{P}^{n-i} 1)}{\cal{P}^{n-i+1} 1}\circ \cal T^{n-i+1}\Big\|_{2}\Big)\\
&\le \frac{C}{\sigma_n^2}\Big(1+\Big\|\sum_{i=1}^{n}\psi_{n-i}\circ \mathcal T^{n-i}\Big\|_2\Big)
\le C\sigma_n^{-1}.
\end{align*}

Next, we verify the claim. It is obvious that
\begin{align*}
&\int\frac{P_{n-i+1}(g_{n-i}\cdot\cal{P}^{n-i} 1)}{\cal{P}^{n-i+1} 1}\circ \cal T^{n-i+1}\rmd m =
\int P_{n-i+1}(g_{n-i}\cdot\cal{P}^{n-i} 1)\rmd m\\
&=\int g_{n-i}\cdot\cal{P}^{n-i} 1\rmd m=\int g_{n-i}\circ \cal T^{n-i}\rmd m=0.
\end{align*}
Since $\inf_{x\in M}\cal P^n1(x)\ge \delta$ for all $n\ge 1$, we have
\[
\Big\|\frac{P_{n-i+1}(g_{n-i}\cdot\cal{P}^{n-i} 1)}{\cal{P}^{n-i+1} 1}\Big\|_\alpha
\le \frac{1}{\delta}\big\|P_{n-i+1}(g_{n-i}\cdot\cal{P}^{n-i} 1)\big\|_\alpha.
\]
Note that $g_{i}\in \mathcal V$, $1\le i\le n$ and $\sup_n\|\cal{P}^{n} 1\|_\infty<\infty$, we have
\[
\Big\|\frac{P_{n-i+1}(g_{n-i}\cdot\cal{P}^{n-i} 1)}{\cal{P}^{n-i+1} 1}\Big\|_\alpha \le C\|g_{n-i}\|_\alpha.
\]
So
\[
\frac{P_{n-i+1}(g_{n-i}\cdot\cal{P}^{n-i} 1)}{\cal{P}^{n-i+1} 1}\in \cal V.
\]
The claim holds.
\end{proof}

\begin{lem}\label{xnw}
Let $p\ge 2$. Then for any $\delta>0$ there exists a constant $C>0$ such that $\mathcal{W}_{p}(X_{n},B)\leq C \sigma_n^{-1/2+\delta}$ for all $n\ge 1$.
\end{lem}

\begin{proof}
The proofs are based on the ideas employed in the stationary case in \cite[Lemma~4.4]{Liu23}. To obtain the convergence rate, we have to produce a bound of $\mathcal{W}_{p}(X_{n},B)$ for fixed $n\ge1$. It suffices to deal with a single row of the array $\{\xi_{n,j},\mathcal{G}_{n,j}, 1\le j\le n\}$.

By the Skorokhod embedding theorem (see Theorem \ref{ske}), there exists a probability space (depending on $n$) supporting a standard Brownian motion, still denoted by $B$ which should not cause confusion, and a sequence of nonnegative random variables $\tau_1,\ldots, \tau_n$ such that for $T_i=\sum_{j=1}^{i}\tau_j$ we have $\sum_{j=1}^{i}\xi_{n,j}=B(T_i)$ with $1\le i\le n$. In particular, we set $T_{0}=0$. Then on this probability space and for this Brownian motion, we aim to show that for any $\delta>0$ there exists a constant $C>0$ such that
\begin{align*}
\Big\|\sup_{t\in[0,1]}|X_{n}(t)-B(t)|\Big\|_{p}\leq C\sigma_n^{-\frac{1}{2}+\delta}.
\end{align*}
Thus the result follows from the definition of the Wasserstein distance.

For ease of exposition when there is no ambiguity, we will write $\xi_j$ and $V_k$ instead of $\xi_{n,j}$ and $V_{n,k}$ respectively. Then by the Skorokhod embedding theorem, we can write \eqref{xn} as
\begin{align}\label{set}
X_{n}(t)=B(T_{k})+\bigg(\frac{tV_{n}-V_{k}}{V_{k+1}-V_{k}}\bigg)\big(B(T_{k+1})-B(T_{k})\big),\quad \hbox{if}~V_{k}\leq tV_{n}<V_{k+1}.
\end{align}

1. We first estimate $|X_{n}-B|$ on the set $\{|T_{n}-1|\ge 1\}$.
 Note that Theorem \ref{ske} (3) implies
\[
T_k-V_k=\sum_{i=1}^{k}\big(\tau_{i}-\mathbb{E}(\tau_{i}|\mathcal{F}_{i-1})\big),\quad 1\le k\le n,
\]
where $\mathcal{F}_{i}$ is the $\sigma$-field generated by all events up to $T_i$ for $1\le i\le n$. Therefore $\{T_k-V_k, \mathcal{F}_{k}, 1\le k\le n\}$ is a martingale. By the conditional Jensen inequality,  $|\mathbb{E}(\tau_{i}|\mathcal{F}_{i-1})|^p\le \mathbb{E}(|\tau_{i}|^p|\mathcal{F}_{i-1})$ for $p>1$. Then by Theorem~\ref{bhi}, we have
\begin{align*}
&\E\big[\max_{1\le k\le n}|T_k-V_{k}|^p\big]\\
\le&C\E\Big[\sum_{i=1}^{n}\E\big[|\tau_i-\mathbb{E}(\tau_{i}|\mathcal{F}_{i-1})|^2\big|\mathcal{F}_{i-1}\big]\Big]^{p/2}\\
&\quad\quad+C\E\Big[\max_{1\le i\le n}|\tau_i-\mathbb{E}(\tau_{i}|\mathcal{F}_{i-1})|^p\Big]\\
\le &C\E\Big[\sum_{i=1}^{n}\mathbb{E}(\tau_{i}^2|\mathcal{F}_{i-1})\Big]^{p/2}
+C\E\Big[\sum_{i=1}^{n}\E\big(|\tau_{i}|^p|\mathcal{F}_{i-1}\big)\Big]\\
\le &C\E\Big[\sum_{i=1}^{n}\mathbb{E}(\xi_{i}^4|\mathcal{G}_{i-1})\Big]^{p/2}
+C\E\Big[\sum_{i=1}^{n}\E\big(|\xi_{i}|^{2p}|\mathcal{G}_{i-1}\big)\Big],
\end{align*}
where the last inequality is based on Theorem~\ref{ske} (4).

For the first term, note that $\{\psi_i\}$ is uniformly bounded, by the argument in the proof of Proposition~\ref{vn1}, we have
\begin{align*}
&\E\Big[\sum_{i=1}^{n}\mathbb{E}(\xi_{i}^4|\mathcal{G}_{i-1})\Big]^{p/2}
\le \frac{\sup_i\|\psi_i\|_\infty^p}{\sigma_n^{2p}}\E\Big[\sum_{i=1}^{n}\mathbb{E}(\psi_{n-i}^2\circ \mathcal T^{n-i}|\mathcal{G}_{i-1})\Big]^{p/2}\\
&\le \frac{C}{\sigma_n^{2p}}\E\Big[\sum_{i=1}^{n}\mathbb{E}(g_{n-i}\circ \mathcal T^{n-i}|\mathcal{G}_{i-1})\Big]^{p/2}+\frac{C}{\sigma_n^{2p}}\E\Big[\sum_{i=1}^{n}\mathbb{E}(\psi_{n-i}^2\circ \mathcal T^{n-i})\Big]^{p/2}\\
&\le \frac{C}{\sigma_n^{2p}}\E\Big[\sum_{i=1}^{n}\frac{P_{n-i+1}(g_{n-i}\cdot\cal{P}^{n-i} 1)}{\cal{P}^{n-i+1} 1}\circ \cal T^{n-i+1}\Big]^{p/2}+\frac{C}{\sigma_n^{2p}}\sigma_n^{p}\\
&\le C\sigma_n^{-p}.
\end{align*}
For the second term,
\begin{align*}
&\E\Big[\sum_{i=1}^{n}\E\big(|\xi_{i}|^{2p}|\mathcal{G}_{i-1}\big)\Big]\\
=&\frac{1}{\sigma_n^{2p}}\E\Big[\sum_{i=1}^{n}\E\big(|\psi_{n-i}\circ \mathcal T^{n-i}|^{2p}|\mathcal{G}_{i-1}\big)\Big]\\
\le &\frac{\sup_i\|\psi_i\|_\infty^{2p-2}}{\sigma_n^{2p}}\E\Big[\sum_{i=1}^{n}|\psi_{n-i}\circ \mathcal T^{n-i}|^{2}\Big]\\
\le &\frac{C}{\sigma_n^{2p}}\sigma_n^{2}=C\sigma_n^{-(2p-2)}.
\end{align*}
Based on the above estimates, we have
\begin{equation}\label{TkVk}
\Big\|\max_{1\le k\le n}|T_k-V_{k}|\Big\|_{p} \le C \sigma_n^{-1}.
\end{equation}
On the other hand, it follows from Proposition~\ref{vn1} that
\begin{align}\label{Vn1}
\|V_{n}-1\|_{p}\le  C \sigma_n^{-1}.
\end{align}
Based on the above estimates, by Chebyshev's inequality we have
\begin{equation}\label{eqq}
\begin{split}
&m(|T_{n}-1|>1)\le \mathbb{E}[|T_{n}-1|^{p}] \\
&\le 2^{p-1}\left\{\mathbb{E}[|T_{n}-V_{n}|^{p}]+\mathbb{E}[|V_{n}-1|^{p}]\right\}\le C\sigma_n^{-p}.
\end{split}
\end{equation}
Note that $\Big\|\sup_{t\in[0,1]}|B(t)|\Big\|_{2p}<\infty$ and by Remark~\ref{mar}, we have
$\Big\|\sup_{t\in[0,1]}|X_{n}(t)|\Big\|_{2p}<\infty$.
Hence, by the H\"{o}lder inequality and \eqref{eqq}, we deduce that
\begin{align*}
I:=&\Big\| 1_{\{|T_{n}-1|>1\}}\sup_{t\in[0,1]}|X_{n}(t)-B(t)|\Big\|_{p}\\
\le &\big(m(|T_{n}-1|>1)\big)^{1/2p}\Big\| \sup_{t\in[0,1]}|X_{n}(t)-B(t)|\Big\|_{2p}\\
\le &\big(m(|T_{n}-1|>1)\big)^{1/2p}\Big(\Big\|\sup_{t\in[0,1]}|X_{n}(t)|\Big\|_{2p}+\Big\|\sup_{t\in[0,1]}|B(t)|\Big\|_{2p}\bigg)\\
\le &C\sigma_n^{-\frac{1}{2}}.
\end{align*}

2. We now estimate $|X_{n}-B|$ on the set $\{|T_{n}-1|\le 1\}$:
\begin{align*}
&\Big\| 1_{\{|T_{n}-1|\le 1\}}\sup_{t\in[0,1]}|X_{n}(t)-B(t)|\Big\|_{p}\\
\le &\Big\| 1_{\{|T_{n}-1|\le 1\}}\sup_{t\in[0,1]}|X_{n}(t)-B(T_{k})|\Big\|_{p}+\Big\|1_{\{|T_{n}-1|\le 1\}}\sup_{t\in[0,1]}|B(T_{k})-B(t)|\Big\|_{p}\\
 =: & I_1 + I_2.
\end{align*}
For $I_1$, it follows from \eqref{set} that
\[
\sup_{t\in[0,1]}|X_{n}(t)-B(T_{k})|\le \max_{0\le k\le n-1}|B(T_{k+1})-B(T_{k})|=\max_{0\le k\le n-1}|\xi_{k+1}|.
\]
Since $\{\psi_n\}$ is uniformly bounded, we have
\begin{align*}\label{zeta}
I_1=&\Big\| 1_{\{|T_{n}-1|\le 1\}}\sup_{t\in[0,1]}|X_{n}(t)-B(T_{k})|\Big\|_{p}\\
\le &\Big\| 1_{\{|T_{n}-1|\le 1\}}\max_{0\le k\le n-1}|\xi_{k+1}|\Big\|_{p}\\
\le &\Big\| \max_{0\le k\le n-1}|\xi_{k+1}|\Big\|_{p}\\
\le &\frac{1}{\sigma_n}\max_{0\le k\le n-1}\big\|\psi_{k}\circ\cal T^{k}\big\|_{\infty}\le C\sigma_n^{-1}.
\end{align*}

3. Finally, we consider $I_{2}$ on the set $\{|T_{n}-1|\le 1\}$. Take $p_1>p$, then it is well known that
\begin{equation}\label{bts}
\mathbb{E}|B(t)-B(s)|^{2p_1}\le c|t-s|^{p_1}, \quad \text{for~ all~} s,t\in [0,2].
\end{equation}
So it follows from Kolmogorov's continuity theorem (see Theorem \ref{kcc}) that for each $0<\gamma<\frac{1}{2}-\frac{1}{2p_1}$, the process $B(\cdot)$ admits a version, still denoted by $B$, such that for almost all $\omega$ the sample path $t\mapsto B(t,\omega)$ is H\"{o}lder continuous with exponent $\gamma$ and
\begin{equation*}
\Big\|\sup_{s,t\in[0,2]\atop s\neq t}\frac{|B(s)-B(t)|}{|s-t|^{\gamma}}\Big\|_{2p_1}< \infty.
\end{equation*}
In particular,
\begin{equation}\label{holder}
\Big\|\sup_{s,t\in[0,2]\atop s\neq t}\frac{|B(s)-B(t)|}{|s-t|^{\gamma}}\Big\|_{2p}< \infty.
\end{equation}

As for $|T_{k}-t|$, by certain calculations (see \cite[Lemma~4.4]{Liu23} for details), we have
\begin{align*}
&\sup_{t\in[0,1]}|T_{k}-t|\le \max_{0\le k\le n-1}\sup_{t\in[\frac{V_{k}}{V_{n}},\frac{V_{k+1}}{V_{n}})}|T_{k}-t|\\
&\le \max_{0\le k\le n}\big|T_{k}-V_{k}\big| + 3\max_{0\le k\le n}\Big|V_{k}-\frac{V_{k}}{V_{n}}\Big| +\max_{0\le k\le n-1}\big|V_{k+1}-V_{k}\big|.
\end{align*}
Note that $T_{0}=V_{0}=0$ and $\gamma\le1$, so
\[
\sup_{t\in[0,1]}|T_{k}-t|^{\gamma}\le \max_{1\le k\le n}\left|T_{k}-V_{k}\right|^{\gamma} + 3^\gamma\max_{1\le k\le n}\Big|V_{k}-\frac{V_{k}}{V_{n}}\Big|^{\gamma} +\max_{0\le k\le n-1}\left|V_{k+1}-V_{k}\right|^{\gamma}.
\]
Hence we have
\begin{align}
&\Big\|\sup_{t\in[0,1]}|T_{k}-t|^{\gamma}\Big\|_{2p}\nonumber\\
\le & \Big\|\max_{1\le k\le n}\big|T_{k}-V_{k}\big|\Big\|_{2\gamma p}^{\gamma} +3^\gamma \Big\|\max_{1\le k\le n}\big|V_{k}-\frac{V_{k}}{V_{n}}\big|\Big\|_ {2\gamma p}^{\gamma}
 +\Big\|\max_{0\le k\le n-1}\big|V_{k+1}-V_{k}\big|\Big\|_ {2\gamma p}^{\gamma}.
\end{align}
For the first term, since $\gamma< \frac{1}{2}$, it follows from \eqref{TkVk} that
\begin{equation}\label{Tgam}
\Big\|\max_{1\le k\le n}|T_{k}-V_{k}|\Big\|_{2\gamma p}^{\gamma}\le C\sigma_n^{-\gamma}.
\end{equation}
For the second term, since $|V_{k}-\frac{V_{k}}{V_{n}}|=V_k|1-\frac1{V_n}|$, we have
\[\max_{1\le k\le n}\Big|V_{k}-\frac{V_{k}}{V_{n}}\Big|=V_n \Big|1-\frac1{V_n}\Big|=|V_{n}-1|.\]
Hence by \eqref{Vn1},
\begin{align}\label{Vgam}
\Big\|\max_{1\le k\le n}\big|V_{k}-\frac{V_{k}}{V_{n}}\big|\Big\|_ {2\gamma p}^{\gamma}
=\big\|V_{n}-1\big\|_{2\gamma p}^{\gamma}\le C\sigma_n^{-\gamma}.
\end{align}
As for the last term, note that for all $1\le k\le n$,
\begin{align*}
|V_{k}-V_{k-1}|&=\mathbb{E}(\xi_{k}^2|\mathcal{F}_{k-1})
=\frac{1}{\sigma_n^2}\mathbb{E}\big(\psi_{n-k}^2\circ \mathcal T^{n-k}|\mathcal{G}_{k-1}\big)\\
&=\frac{1}{\sigma_n^2}\cdot\frac{P_{n-k+1}(\psi_{n-k}^2\cdot\cal{P}^{n-k} 1)}{\cal{P}^{n-k+1} 1}\circ \cal T^{n-k+1}.
\end{align*}
Since $\sup_n\max_{k\le n}\|\frac{P_{n-k+1}(\psi_{n-k}^2\cdot\cal{P}^{n-k} 1)}{\cal{P}^{n-k+1} 1}\|_\infty<\infty$, we have
\begin{align}\label{mgam}
&\Big\|\max_{0\le k\le n-1}\big|V_{k+1}-V_{k}\big|\Big\|_ {2\gamma p}^{\gamma}=\frac{1}{\sigma_n^{2\gamma}}\Big\|\max_{1\le k\le n}\big|\frac{P_{n-k+1}(\psi_{n-k}^2\cdot\cal{P}^{n-k} 1)}{\cal{P}^{n-k+1} 1}\circ \cal T^{n-k+1}\big|\Big\|_{2\gamma p}^{\gamma}
\le C\sigma_n^{-2\gamma}.
\end{align}
Based on the above estimates \eqref{Tgam}--\eqref{mgam}, we have
\begin{align}\label{te}
\Big\|\sup_{t\in[0,1]}|T_{k}-t|^{\gamma}\Big\|_{2p}\le C\sigma_n^{-\gamma}.
\end{align}

On the set $\{|T_{n}-1|\le 1\}$, note that
\[
\sup_{t\in[0,1]}|B(T_{k})-B(t)|\le \Big[\sup_{s,t\in[0,2]\atop s\neq t}\frac{|B(s)-B(t)|}{|s-t|^{\gamma}}\Big]\Big[\sup_{t\in[0,1]}|T_{k}-t|^{\gamma}\Big].
\]
Since $0<\gamma<\frac{1}{2}-\frac{1}{2p_1}$, by the H\"{o}lder inequality and \eqref{holder}, \eqref{te}, we have
\begin{align*}
I_2=&\Big\|1_{\{|T_{n}-1|\le 1\}}\sup_{t\in[0,1]}|B(T_{k})-B(t)|\Big\|_{p}\\
\le &\Big\|\Big[\sup_{s,t\in[0,2]\atop s\neq t}\frac{|B(s)-B(t)|}{|s-t|^{\gamma}}\Big]\Big[\sup_{t\in[0,1]}|T_{k}-t|^{\gamma}\Big]\Big\|_{p}\\
\le &\Big\|\sup_{s,t\in[0,2]\atop s\neq t}\frac{|B(s)-B(t)|}{|s-t|^{\gamma}}\Big\|_{2p}\Big\|\sup_{t\in[0,1]}|T_{k}-t|^{\gamma}\Big\|_{2p}\\
\le & C \sigma_n^{-\gamma}.
\end{align*}
Note that $p_1$ can be taken arbitrarily large in \eqref{bts}, which implies that $\gamma$ can be chosen sufficiently close to $\frac{1}{2}$. So for any $\delta>0$, we can choose $p_1$ large enough such that $I_2\le C\sigma_n^{-\frac{1}{2}+\delta}$. The result now follows from the above estimates for $I,I_1$ and $I_2$.
\end{proof}

{\bf Step 3. Estimation of the Wasserstein convergence rate between $M_n$ and $X_n$.}

\vskip3mm

Define a continuous transformation $g:C[0,1]\rightarrow C[0,1]$ by $g(u)(t):=u(1)-u(1-t)$.

\begin{lem}\label{wnxn}
Let $p\ge2$. Then for any $\delta>0$, there exists a constant $C>0$ such that for all $n\ge 1$, $\mathcal{W}_{p}(g\circ M_{n}, X_{n})\leq C\sigma_n^{-\frac{1}{2}+\delta}$.
\end{lem}

\begin{proof}
For $1\le j\le n$, we recall that $\alpha_{j}^{2}=\sum_{i=1}^{j}\int\psi_{n-i}^2\circ \mathcal T^{n-i}\rmd m$. Then $\alpha_{n}^{2}=\sigma_n^2$. We define
\[
\widetilde M_n(t):=\frac{1}{\sigma_n}\bigg[\sum_{i=1}^{l}\psi_{n-i}\circ\cal T^{n-i}+\frac{t\alpha_{n}^{2}-\alpha_{l}^{2}}{\alpha_{l+1}^{2}-\alpha_{l}^{2}}\psi_{n-l-1}\circ\cal T^{n-l-1}\bigg], \quad \textrm{if}~~~~ \alpha_{l}^{2}\leq t\alpha_{n}^{2}<\alpha_{l+1}^{2}.
\]

1. We first estimate $\big\|\sup_{t\in[0,1]}|\widetilde M_n(t)-X_n(t)|\big\|_{p}$. By the Skorokhod embedding theorem in Lemma~\ref{xnw}, we know that there exists a sequence of nonnegative random variables $T_1,\ldots, T_n$ such that $\sum_{j=1}^{i}\frac{1}{\sigma_n}\psi_{n-j}\circ\cal T^{n-j}=B(T_i)$ with $1\le i\le n$.
Define a continuous process $Y_n(t)\in C[0,1]$,
\[Y_n(t):=B(tT_n), \quad t\in [0,1].\]
Then
\begin{align}\label{WYX}
\Big\|\sup_{t\in[0,1]}|\widetilde M_n(t)-X_n(t)|\Big\|_{p}\le \Big\|\sup_{t\in[0,1]}|\widetilde M_n(t)-Y_n(t)|\Big\|_{p}+\Big\|\sup_{t\in[0,1]} |X_n(t)-Y_n(t)|\Big\|_{p}.
\end{align}
On the set $\{|T_n-1|>1/2\}$, by the first step in the proof of Lemma~\ref{xnw}, we have
\begin{align*}
&\Big\|1_{\{|T_{n}-1|>\frac{1}{2}\}}\sup_{t\in[0,1]}|\widetilde M_n(t)-X_n(t)|\Big\|_{p}\\
\le &\big(m(|T_{n}-1|>1/2)\big)^{1/2p}\Big\|\sup_{t\in[0,1]}|\widetilde M_n(t)-X_n(t)|\Big\|_{2p}\\
\le &C\sigma_n^{-1/2}.
\end{align*}

In the following, we aim to estimate \eqref{WYX} on the set $\{|T_n-1|\le1/2\}$. Denote a set $E_n:=\{\max_{1\le j\le n}\big|\frac{T_j}{T_n}-\frac{\alpha_j^2}{\alpha_n^2}\big|\le\eps;
\max_{1\le j\le n}\big|\frac{\alpha_{j}^2}{\alpha_n^2}-\frac{\alpha_{j-1}^2}{\alpha_n^2}\big|\le\eps\}$.
Then
\begin{align*}
&\Big\|\sup_{t\in[0,1]}|\widetilde M_n(t)-Y_n(t)|\Big\|_{p}\\
\le& \Big\|1_{E_n^c}\sup_{t\in[0,1]}|\widetilde M_n(t)-Y_n(t)|\Big\|_{p}\\
&\quad\quad\quad+\Big\|1_{E_n}\sup_{t\in[0,1]}|\widetilde M_n(t)-Y_n(t)|\Big\|_{p}\\
=:&I_1+I_2.
\end{align*}

To deal with $I_1$, we first estimate that for any $\kappa\ge 1$,
\begin{align*}
&\E\Big[\max_{1\le j\le n}\big|\frac{\alpha_j^2}{\alpha_n^2}-\frac{V_{n,j}}{V_{n,n}}\big|^{\kappa p}\Big]\\
\le &2^{\kappa p-1}\Big(\E\big[\max_{1\le j\le n}\big|\frac{\alpha_j^2}{\alpha_n^2}-V_{n,j}\big|^{\kappa p}\big]+\E\big[\max_{1\le j\le n}\big|V_{n,j}-\frac{V_{n,j}}{V_{n,n}}\big|^{\kappa p}\big]\Big)\\
\le &2^{\kappa p}\frac{1}{\alpha_n^{2\kappa p}}\E\big[\max_{1\le j\le n}\big|\alpha_j^2-V_{n,j}\cdot \alpha_n^2\big|^{\kappa p}\big],
\end{align*}
where the last inequality is due to
\[\max_{1\le j\le n}\Big|V_{n,j}-\frac{V_{n,j}}{V_{n,n}}\Big|=V_{n,n}\Big|1-\frac1{V_{n,n}}\Big|=|V_{n,n}-1|\le \max_{1\le j\le n}\big|V_{n,j}-\frac{\alpha_j^2}{\alpha_n^2}\big|.\]
Then it follows from the proof of Proposition~\ref{vn1} and Proposition~\ref{mom} that
\begin{align*}
&\E\Big[\max_{1\le j\le n}\big|\alpha_j^2-V_{n,j}\cdot \alpha_n^2\big|^{\kappa p}\Big]\\
=&\E\Big[\max_{1\le j\le n}\big|\sum_{i=1}^{j}\E(\psi_{n-i}^2\circ \mathcal T^{n-i}|\mathcal{G}_{n,i-1})-\sum_{i=1}^{j}\E(\psi_{n-i}^2\circ \mathcal T^{n-i})\big|^{\kappa p}\Big]\\
=&\E\Big[\max_{1\le j\le n}\big|\sum_{i=1}^{j}\E\big(g_{n-i}\circ \mathcal T^{n-i}\big|\mathcal{G}_{n,i-1}\big)\big|^{\kappa p}\Big]\\\nonumber
=&\E\Big[\max_{1\le j\le n}\big|\sum_{i=1}^{j}\frac{P_{n-i+1}(g_{n-i}\cdot\cal{P}^{n-i} 1)}{\cal{P}^{n-i+1} 1}\circ \cal T^{n-i+1}\big|^{\kappa p}\Big]\\
\le &C\sigma_n^{\kappa p}.
\end{align*}
So
\[
\E\Big[\max_{1\le j\le n}\big|\frac{\alpha_j^2}{\alpha_n^2}-\frac{V_{n,j}}{V_{n,n}}\big|^{\kappa p}\Big]\le C\sigma_n^{-\kappa p}.
\]
Also, by \eqref{TkVk} and the assumption $|T_n-1|\le 1/2$,
\[
\E\Big[\max_{1\le j\le n}\big|\frac{T_j}{T_n}-\frac{V_{n,j}}{V_{n,n}}\big|^{\kappa p}\Big]\le C\E\Big[\max_{1\le j\le n}\big|T_j-V_{n,j}\big|^{\kappa p}\Big]\le C\sigma_n^{-\kappa p}.
\]
Hence, by Chebyshev's inequality, we have for any $\kappa\ge 1$,
\begin{align}\label{Talp}
&\nonumber m(\max_{1\le j\le n}\big|\frac{T_j}{T_n}-\frac{\alpha_j^2}{\alpha_n^2}\big|>\eps)\le \frac{\E\big[\max_{1\le j\le n}\big|\frac{T_j}{T_n}-\frac{\alpha_j^2}{\alpha_n^2}\big|^{\kappa p}\big]}{\eps^{\kappa p}}\\\nonumber
\le&\frac{2^{\kappa p-1}(\E\big[\max_{1\le j\le n}\big|\frac{T_j}{T_n}-\frac{V_{n,j}}{V_{n,n}}\big|^{\kappa p}\big]+\E\big[\max_{1\le j\le n}\big|\frac{\alpha_j^2}{\alpha_n^2}-\frac{V_{n,j}}{V_{n,n}}\big|^{\kappa p}\big])}{\eps^{\kappa p}}\\
\le &C \eps^{-\kappa p}\sigma_n^{-\kappa p}.
\end{align}
Similarly,
\begin{align}\label{alpjn}
m(\max_{1\le j\le n}\big|\frac{\alpha_{j}^2}{\alpha_n^2}-\frac{\alpha_{j-1}^2}{\alpha_n^2}\big|>\eps)
\le C \eps^{-\kappa p}\sigma_n^{-2\kappa p}.
\end{align}
Note that by Remark~\ref{mar}, $\Big\|\sup_{t\in[0,1]}|\widetilde M_{n}(t)|\Big\|_{2p}<\infty$ and $\Big\|\sup_{t\in[0,1]}|Y_n(t)|\Big\|_{2p}<\infty$.
Then by the H\"older inequality, \eqref{Talp} and \eqref{alpjn}, we have
\begin{align*}
I_1&\le \Big(m(\max_{1\le j\le n}\big|\frac{T_j}{T_n}-\frac{\alpha_j^2}{\alpha_n^2}\big|>\eps)+m(\max_{1\le j\le n}\big|\frac{\alpha_{j}^2}{\alpha_n^2}-\frac{\alpha_{j-1}^2}{\alpha_n^2}\big|>\eps)\Big)^{\frac{1}{2p}}\\
&\quad\quad\times\Big(\Big\|\sup_{t\in[0,1]}|\widetilde M_{n}(t)|\Big\|_{2p}
+\Big\|\sup_{t\in[0,1]}|Y_n(t)|\Big\|_{2p}\Big)\\
&\le C\eps^{-\frac{\kappa}{2}}\sigma_n^{-\frac{\kappa}{2}}.
\end{align*}

As for the term $I_2$, by the relation $\sum_{j=1}^{i}\frac{1}{\sigma_n}\psi_{n-j}\circ\cal T^{n-j}=B(T_i)$ and
$\frac{1}{2}\le T_n\le \frac{3}{2}$, we have
\begin{align*}
I_2&=\Big\|\max_{1\le l\le n}\sup_{\frac{\alpha_{l}^2}{\alpha_n^2}\le t<\frac{\alpha_{l+1}^2}{\alpha_n^2}}|\widetilde M_n(t)-Y_n(t)|1_{E_n}\Big\|_{p}\\
&\le \Big\|\max_{1\le l\le n}\sup_{\frac{\alpha_{l}^2}{\alpha_n^2}\le t<\frac{\alpha_{l+1}^2}{\alpha_n^2}}| B(T_l)-B(tT_n)|1_{E_n}\Big\|_{p}+O(\sigma_n^{-1})\\
&\le \Big\|\sup_{|u-v|<3\eps}| B(u)-B(v)\Big\|_{p}+O(\sigma_n^{-1})\\
&\le C\eps^{1/2}+C\sigma_n^{-1}.
\end{align*}
Hence
\[
\Big\|\sup_{t\in[0,1]}|\widetilde M_n(t)-Y_n(t)|\Big\|_{p}\le C\eps^{-\frac{\kappa}{2}}\sigma_n^{-\frac{\kappa}{2}}+C\eps^{1/2}+C\sigma_n^{-1}.
\]
Taking $\eps=\sigma_n^{-\frac{\kappa}{1+\kappa}}$, we obtain that
\[
\Big\|\sup_{t\in[0,1]}|\widetilde M_n(t)-Y_n(t)|\Big\|_{p}\le C\sigma_n^{-\frac{\kappa}{2(1+\kappa)}}.
\]
Since $\kappa$ can be large enough, for any $\delta>0$, there exists $\kappa\ge 1$ such that
\[
\Big\|\sup_{t\in[0,1]}|\widetilde M_n(t)-Y_n(t)|\Big\|_{p}\le C\sigma_n^{-\frac{1}{2}+\delta}.
\]
By the same arguments, we can also obtain that
\[
\Big\|\sup_{t\in[0,1]}|X_n(t)-Y_n(t)|\Big\|_{p}\le C\sigma_n^{-\frac{1}{2}+\delta}.
\]
So
\[
\Big\|\sup_{t\in[0,1]}|\widetilde M_n(t)-X_n(t)|\Big\|_{p}\le C\sigma_n^{-\frac{1}{2}+\delta}.
\]

2. We now estimate $\big\|\sup_{t\in[0,1]}|g\circ M_{n}(t)-\widetilde M_n(t)|\big\|_{\infty}$. Note that
\begin{align*}
&g\circ M_{n}(t)=M_n(1)-M_n(1-t)\\
&=\frac{1}{\sigma_n}\sum_{i=0}^{n-1}\psi_{i}\circ\cal T^{i}-\frac{1}{\sigma_n}\sum_{i=0}^{r_n(1-t)-1}\psi_{i}\circ\cal T^{i}+F_n(t)\\
&=\frac{1}{\sigma_n}\sum_{i=r_n(1-t)}^{n-1}\psi_{i}\circ\cal T^{i}+F_n(t)\\
&=\frac{1}{\sigma_n}\sum_{i=1}^{n-r_n(1-t)}\psi_{n-i}\circ\cal T^{n-i}+F_n(t),
\end{align*}
where $\|F_n(t)\|_{\infty}\le \sigma_n^{-1}\max_{0\le i\le n-1}\|\psi_i\|_{\infty}\le C\sigma_n^{-1}$.

To compare $n-r_n(1-t)$ with $l_n(t)$, we first find that
\[\sigma_{r_n(1-t)-1}^2<(1-t)\sigma_n^2\le\sigma_{r_n(1-t)}^2.\]
Since $\sigma_n^2=\alpha_n^2$, we have
\[\alpha_n^2-\alpha_{n-r_n(1-t)+1}^2<(1-t)\alpha_n^2\le \alpha_n^2-\alpha_{n-r_n(1-t)}^2,\]
i.e.
\[\alpha_{n-r_n(1-t)}^2\le t\alpha_n^2<\alpha_{n-r_n(1-t)+1}^2.\]
By the definition of $l_n(t)$, we also have
$\alpha_{l_n(t)}^2\le t\alpha_n^2<\alpha_{l_n(t)+1}^2$. So $l_n(t)=n-r_n(1-t)$.
Hence
\[
\Big\|\sup_{t\in[0,1]}|g\circ M_{n}(t)-\widetilde M_n(t)|\Big\|_{\infty}\le C\frac{1}{\sigma_n}\max_{0\le i\le n-1 }\big\|\psi_i\circ \cal T^i\big\|_{\infty}\le C\sigma_n^{-1}.
\]

3. Combining the above estimates, by the definition of Wasserstein distance, we obtain that for all $n\ge1$,
\begin{align*}
\mathcal{W}_{p}(g\circ M_{n}, X_{n})&\le \mathcal{W}_{p}(g\circ M_{n}, \widetilde M_n)+\mathcal{W}_{p}(\widetilde M_n, X_{n})\\
&\le C\sigma_n^{-1}+C\sigma_n^{-\frac{1}{2}+\delta}\le C\sigma_n^{-\frac{1}{2}+\delta}
\end{align*}
with $\delta$ sufficiently small.
\end{proof}

\begin{proof}[Proof of Theorem \ref{thnon}]
Recall that $g:C[0,1]\rightarrow C[0,1]$ is a continuous transformation defined by $g(u)(t)=u(1)-u(1-t)$. We note that $g\circ g= Id$ and $g$ is Lipschitz with Lip $g$ $\le 2$. It follows from the Lipschitz mapping theorem (see \cite[Proposition~2.4]{Liu23}) that
\[
\mathcal{W}_{p}(M_{n},B)= \mathcal{W}_{p}(g(g\circ M_{n}),g(g\circ B))\le 2\mathcal{W}_{p}(g\circ M_{n},g\circ B).
\]
Since $g(B)=_d B$, by Lemmas \ref{xnw} and \ref{wnxn}, for $p\ge 2$ we have
\begin{align*}
&\mathcal{W}_{p}(g\circ M_{n},g\circ B)\le\mathcal{W}_{p}(g\circ M_{n}, X_n)+\mathcal{W}_{p}(X_n, B)\\
&\le C\sigma_n^{-\frac{1}{2}+\delta}+C\sigma_n^{-\frac{1}{2}+\delta}\le C\sigma_n^{-\frac{1}{2}+\delta}\asymp C\Sigma_n^{-\frac{1}{2}+\delta}
\end{align*}
with $\delta$ sufficiently small and $n\ge1$.

Finally, by Lemma~\ref{wmn}, we conclude that
\begin{align*}
\mathcal{W}_{p}(W_{n},B)&\le \mathcal{W}_{p}(W_{n},M_n)+\mathcal{W}_{p}(M_{n},B)\\
&\le C\Sigma_n^{-\frac{1}{2}+\delta}+C\Sigma_n^{-\frac{1}{2}+\delta}\le C\Sigma_n^{-\frac{1}{2}+\delta}
\end{align*}
with $\delta$ sufficiently small and $n\ge1$.
\end{proof}

\section{Applications of Theorem~\ref{thnon}}
In this section, we introduce a class of systems investigated in~\cite{Haydn17} as concrete examples to which the Wasserstein convergence
rate in the invariance principle (Theorem~\ref{thnon}) applies. In order to guarantee the conditions in Theorem~\ref{thnon}, a few assumptions are needed. For the convenience, we recall the assumptions first and then provide a list of examples. We refer to~\cite{Conze07} and \cite[Section~7]{Haydn17} for more details.

We say that a transfer operator $P$ is exact if $\lim_{n\to \infty}\|P^nv\|_1=0$, $\forall v\in \cal V$ with mean zero (w.r.t. Lebesgue measure). We define a distance
between two transfer operators $P$ and $Q$ by taking
\[
d(P,Q)=\sup_{v\in \cal V, \|v\|_{\alpha}\le 1}\|Pv-Qv\|_1.
\]
In the following, the maps we consider in $\mathcal F$ will be close to a given map $T_0$. Roughly speaking, the word ``close" means that $d(P_n, P_0)\to 0$, as $n\to \infty$. We will give a detailed description below.

One of the basic assumptions is a ``quasi-compactness" condition:\\
{\bf Uniform Doeblin-Fortet-Lasota-Yorke inequality (DFLY).} Given the family $\mathcal F$, there exist constants $A,B<\infty$, $\gamma\in(0,1)$ such that for any $n\in \N$, any sequence of operators $P_1, P_2,\ldots, P_n$ corresponding to maps chosen from $\mathcal F$ and
any $v\in \cal V$, we have
\begin{align}\label{dfly}
\|P_n\circ P_{n-1}\circ\cdots\circ P_1v\|_\alpha\le A\gamma^n\|v\|_\alpha+B\|v\|_1.
\end{align}
In particular, the bound~\eqref{dfly} is valid when it is applied to $P_0^n$. Namely, we require:\\
{\bf Exactness property (Exa).} The operator $P_0$ has a spectral gap, which implies that there exist constants $C<\infty$, $\gamma_0\in(0,1)$ such that for any $n\ge 1$ and $v\in \cal V$,
\[
\|P_0^nv\|_\alpha\le C\gamma_0^n\|v\|_\alpha.
\]

By the definition of $\|\cdot\|_\alpha$, $\|v\|_\infty\le C_1\|v\|_\alpha$. We know  that $\|P_0^nv\|_1\le C\|P_0^nv\|_\alpha\to 0$. So the transfer operator $P_0$ is exact. To verify the property (DEC), a useful criterion was given in \cite[Proposition~2.10]{Conze07}. It says that if $P_0$ is exact, then there exists
$\delta_0>0$ such that the set $\{P: d(P,P_0)<\delta_0\}$ satisfies the property (DEC).\\
{\bf Lipschitz continuity property (Lip).} Assume that the maps (and their transfer operators) are  parametrized by a sequence of numbers $\eps_k$, $k\in\N$, such that $\lim_{k\to\infty}\eps_k=\eps_0$ $(P_{\eps_0}=P_0)$. We assume that there exists a constant $C_1<\infty$ such that
\[
d(P_{\eps_k}, P_{\eps_j})\le C_1|\eps_k-\eps_j|, \quad \hbox{for~ all~} k,j\ge 0.
\]
In the following, the maps we consider are restricted to a  subclass of maps; that is $\{T_{\eps_k}: |\eps_k-\eps_0|< C_1^{-1}\delta_0\}$. Then the maps in this subclass satisfy the (DEC) condition. Besides, we also need a quantitative assumption:\\
{\bf Convergence property (Conv).}  There exist constants $C_2<\infty$, $\kappa>0$ such that
\[
|\eps_n-\eps_0|\le C_2\frac{1}{n^\kappa} \quad\forall n\ge 1.
\]

Finally, we also require:\\
{\bf Positivity property (Pos).} The density $h$ for the limiting map $T_0$ is strictly positive. Namely,
\[
\inf_x h(x)>0.
\]

The above properties can be summarized to obtain the following result.
\begin{lem}\cite[Lemma~7.1]{Haydn17}
Assume the assumptions (Exa), (Lip), (Conv) and (Pos) are satisfied. If $v$ is not a coboundary for $T_0$, then $\Sigma_n^2/n$ converges as
$n\to\infty$ to $\Sigma^2$ which is given by
\[
\Sigma^2=\int \hat P[Gv-\hat PGv]^2(x)h(x)\rmd x,
\]
where $\hat Pv=\frac{P_0(hv)}{h}$ is the normalized transfer operator of $T_0$ and $Gv=\sum_{k\ge 0}\frac{P_0^k(hv)}{h}$.
\end{lem}

\subsection{$\beta$-transformations}
Let $\beta>1$ and denote by $T_\beta(x)=\beta x\mod 1$ the $\beta$-transformation on the unit interval $M=[0,1]$. Let $c>0$ and $\beta_k$ be real number such
that $\beta_k\ge 1+c$, $k\ge 1$. Then $\{T_{\beta_k}: k\ge1\}$ is the family of maps we want to consider here. We take the functional space $\cal V$ to be the Banach space of bounded variation functions with norm $\|\cdot\|_{BV}$. The property (DEC) and (MIN) were proved in \cite[Theorem~3.4(c)]{Conze07} and  \cite[Proposition~4.3]{Conze07}, respectively.
The invariant density $h$ of $T_\beta$ is bounded below, and the continuity (Lip) was introduced in \cite[Lemma~3.9]{Conze07}. Then by Theorem~\ref{thnon}, we obtain
\begin{thm}
Assume that $|\beta_n-\beta|\le n^{-\theta}$, $\theta>1/2$. Let $v\in BV$ be such that $m(hv)=0$, where $m$ is the Lebesgue measure and $v$ is not a coboundary for $T_\beta$. Let $W_n$ be defined by \eqref{wnt} and $B$ a standard Brownian motion. Then for any $\delta>0$, there exists a constant $C>0$ such that
$\mathcal{W}_{p}(W_{n},B)\leq C \Sigma_n^{-\frac{1}{2}+\delta}$ for all $n\ge 1$ and $p\ge 2$.
\end{thm}
\subsection{Piecewise expanding map on the interval}
Let $T$ be a piecewise uniformly expanding map on the unit interval $M=[0,1]$. We assume that $T$ is locally injective on the open intervals $A_k$, $k=1,\ldots,m$, that give a partition $\cal A=\{A_k: k\}$ of the unit interval (up to zero measure sets). The map $T$ is $C^2$ on each $A_k$ and has a $C^2$ extension to the boundaries. Moreover, there exist $\Lambda>1$, $C<\infty$ such that $\inf_{x\in M}|DT(x)|\ge \Lambda$ and $\sup_{x\in M} \left|\frac{D^2T(x)}{DT(x)}\right|\le C$.

The family of maps we consider here are constructed with local additive noise starting from $T$. On each interval $A_k$, we define $T_\eps=T(x)+\eps$, where $|\eps|<1$ and we restrict the values of $\eps$ such that the images $T_\eps A_k$, $k=1,\ldots,m$ are strictly included in $[0,1]$. We also suppose that there exists an element $A_\omega\in \cal A$ such that\\
(i) $A_\omega\subset T_\eps A_k$ for all $T_\eps$ and $k=1,\ldots,m$; \\
(ii) The map $T$ sends $A_\omega$ to the whole unit interval. In particular, there exists $1>L'>0$ such that for all $T_\eps$ and $k=1,\ldots,m$, $|T_\eps(A_\omega)\cap A_k|>L'$.

We take the functional space $\cal V$ to be the Banach space of bounded variation functions with norm $\|\cdot\|_{BV}$. It follows from \cite[Lemma~7.5]{Haydn17} that the maps $T_\eps$ satisfy the conditions (DFLY), (MIN), (Pos) and (Lip). Hence the variance $\Sigma_n^2$ grows linearly and the standard ASIP holds with variance $\Sigma^2$ by \cite[Theorem~7.6]{Haydn17}. Further, by Theorem~\ref{thnon}, we obtain
\begin{thm}\label{c2}
Let $T$ be a map of the unit interval defined above and such that it has only one absolutely continuous invariant measure, which is also mixing. Assume that $\{T_{\eps_k}\}$ is the sequence of maps, where the sequence $\{\eps_k\}_{k\ge 1}$ satisfies $|\eps_k|\le k^{-\theta}$, $\theta>1/2$. If $v\in BV$ is not a coboundary for $T$, then for any $\delta>0$, there exists a constant $C>0$ such that
$\mathcal{W}_{p}(W_{n},B)\leq C \Sigma_n^{-\frac{1}{2}+\delta}$ for all $n\ge 1$ and $p\ge 2$.
\end{thm}

\begin{rem}
We can also consider multidimensional piecewise expanding maps investigated in \cite{AFLV11,AFV15,HNVZ13,S00}. In this case, we take the functional space $\cal V$ to be the space of quasi-H\"older functions. Then Theorem~\ref{c2} also holds. We refer to Section~7.3.2 in \cite{Haydn17} for more details.
\end{rem}

\subsection{Covering maps: A general class}
We now present a more general class of examples which were introduced in \cite{BV13}. As before the maps we consider here will be constructed around a given map $T:M\to M$ with $M=[0,1]$. We take the functional space $\cal V$ to be the Banach space of bounded variation functions with norm $\|\cdot\|_{BV}$. Now we introduce such a initial map $T$.

{\bf(H1)} There exists a partition $\cal A=\{A_i\}_{i=1}^{m}$ of $M$, which consists of pairwise disjoint intervals $A_i$. Let
$\bar A_i:=[c_{i,0},c_{i+1,0}]$ and there exists $\delta>0$ such that $T_{i,0}:=T|_{(c_{i,0},c_{i+1,0})}$ is $C^2$ and extends to a $C^2$ function
$\bar T_{i,0}$ on a neighbourhood $[c_{i,0}-\delta,c_{i+1,0}+\delta]$ of $\bar A_i$.

{\bf(H2)} There exists $\beta_0<\frac{1}{2}$ such that $\inf_{x\in I\setminus \cal C_0}|T'(x)|\ge \beta_0^{-1}$, where $\cal C_0=\{c_{i,0}\}_{i=1}^{m}$.

Next, we construct the perturbed map $T_\eps$ in the following way. Each map $T_\eps$ has a partition $\{A_{i,\eps}\}_{i=1}^{m}$ of $M$, which consists of pairwise disjoint intervals $A_{i,\eps}$,
$\bar A_{i,\eps}:=[c_{i,\eps},c_{i+1,\eps}]$ such that

(i) for each $i$ we have $[c_{i,0}+\delta,c_{i+1,0}-\delta]\subset [c_{i,\eps},c_{i+1,\eps}]\subset[c_{i,0}-\delta,c_{i+1,0}+\delta]$; whenever
$c_{1,0}=0$ or $c_{m+1,0}=1$, we do not move them with $\delta$. In this way we establish a one-to-one correspondence between the unperturbed and the perturbed boundary points of $A_i$ and $A_{i,\eps}$. (The quantity $\delta$ is from the assumption (H1) above.)

(ii) The map $T_\eps$ is locally injective over the closed intervals $\bar A_{i,\eps}$, of class $C^2$ in their interiors, and expanding with
$\inf_{x}|T'_\eps(x)|> 2$. Moreover, if $c_{i,0}$ and $c_{i,\eps}$ are two (left or right) corresponding points, we assume that there exists $\sigma>0$ such that $\forall \eps>0$, $\forall i=1,\ldots,m$ and
$\forall x\in [c_{i,0}-\delta,c_{i+1,0}+\delta]\cap\bar A_{i,\eps}$, we have
\begin{align}\label{c}
|c_{i,0}-c_{i,\eps}|\le \sigma
\end{align}
and
\begin{align}\label{T12}
|\bar T_{i,0}(x)-T_{i,\eps}(x)|\le \sigma.
\end{align}

We note that the assumption (H2), more precisely the fact that $\beta_0^{-1}$ is strictly bigger than $2$ instead of $1$, is sufficient to get the uniform Doeblin-Fortet-Lasota-Yorke inequality
(DFLY), as explained in Section~4.2 of \cite{GHW11}. In order to deal with the lower bound condition (MIN), we need to require the following condition. We refer to \cite[Section~2.6]{AR16} or \cite[Section~7.4]{Haydn17} for more details.\\
{\bf Covering property.} There exist $n_0$ and $N(n_0)$ such that:\\
(i) the partition into sets $A_{k_1,\ldots,k_{n_0}}^{\eps_1,\cdots,\eps_{n_0}}$ has diameter less than $\frac{1}{2au}$, where we use the notation
$A_{i,\eps_k}$ to denote the $i$ domain of injectivity of the map $T_{\eps_k}$, and
\[
A_{k_1,\ldots,k_{n}}^{\eps_1,\cdots,\eps_{n}}:=T_{k_1,\eps_1}^{-1}\circ\cdots\circ T_{k_{n-1},\eps_{n-1}}^{-1}A_{k_n,\eps_n}\cap\cdots\cap
T_{k_1,\eps_1}^{-1}A_{k_2,\eps_2}\cap A_{k_1,\eps_1}.
\]
(ii) For any sequence $\eps_1,\cdots,\eps_{N(n_0)}$ and $k_1,\ldots,k_{n_0}$ we have
\[
T_{\eps_{N(n_0)}}\circ \cdots \circ T_{\eps_{n_0}+1}A_{k_1,\ldots,k_{n_0}}^{\eps_1,\cdots,\eps_{n_0}}=M.
\]
Meanwhile, the (Pos) condition also follows from the above covering condition. As for the continuity (Lip), we can extend the continuity for the expanding maps of the intervals to the general case if we can get the following bounds:
\begin{align}\label{TDT}
\left.
\begin{matrix}
|T_{\eps_1}^{-1}(x)-T_{\eps_2}^{-1}(x)|\\
|DT_{\eps_1}(x)-DT_{\eps_2}(x)|\\
\end{matrix}
\right\}=O(|\eps_1-\eps_2|),
\end{align}
where the point $x$ is in the same domain of injective of the maps $T_{\eps_1}$ and $T_{\eps_2}$, the comparison of the same functions and
derivative in two different points being controlled by the condition \eqref{c}. The bounds \eqref{TDT} follow easily by adding to \eqref{c},
\eqref{T12} the further assumptions that $\sigma=O(\eps)$ and requiring a continuity condition for derivatives like \eqref{T12} and with
$\sigma$ again being order of $\eps$.

Combining the above statements, we obtain
\begin{thm}
 Let $T:M\to M$ be a map defined above. Assume that $\{T_{\eps_k}\}$ is the sequence of maps satisfying the above conditions, and the sequence $\{\eps_k\}_{k\ge 1}$ satisfies $|\eps_k|\le k^{-\theta}$, $\theta>1/2$. If $v$ is not a coboundary for $T$, then for any $\delta>0$, there exists a constant $C>0$ such that
$\mathcal{W}_{p}(W_{n},B)\leq C \Sigma_n^{-\frac{1}{2}+\delta}$ for all $n\ge 1$ and $p\ge 2$.
\end{thm}
\appendix

\section{}

\begin{thm}[Kolmogorov continuity criterion \cite{K97}]\label{kcc}
Let $X=\{X(t), t\in [0,T] \}$ be an n-dimensional stochastic process such that
\begin{align*}
\E|X(t)-X(s)|^\beta\le C|t-s|^{1+\alpha}
\end{align*}
for constants $\beta, \alpha>0$, $C\ge 0$ and  for all $0\le s,t \le T$. Then $X$ has a continuous version $\widetilde{X}$.

Further for each $0<\gamma<\frac{\alpha}{\beta}$, there exists a positive random variable $K(\omega)$ with $\E(K^{\beta})<\infty$ such that
\[
|\widetilde{X}(t,\omega)- \widetilde{X}(s,\omega)|\le K(\omega)|s-t|^{\gamma},\quad \hbox{for~every~} s,t \in [0,T]
\]
holds for almost all $\omega$.
\end{thm}

\begin{thm}[Skorokhod embedding theorem \cite{HH80}]\label{ske}
Let $\{S_n=\sum_{i=1}^{n}X_{i},\mathcal{F}_{n}, n\ge 1\}$ be a zero-mean, square-integrable martingale. Then there exist a probability space supporting a (standard) Brownian motion $W$ and a sequence of nonnegative variables $\tau_{1}, \tau_{2},\ldots$ with the following properties: if $T_{n}=\sum_{i=1}^{n}\tau_{i}$, $S'_{n}=W(T_n)$, $X'_1=S'_1$, $X'_n=S'_{n}-S'_{n-1}$ for $n\ge 2$, and $\mathcal{B}_{n}$ is the $\sigma$-field generated by $S'_1,\ldots, S'_{n}$ and $W(t)$ for $0\le t\le T_n$, then
\begin{enumerate}
\item $\{S_{n}, n\ge 1\}=_{d} \{S'_n, n\ge 1\}$;
\item $T_n$ is a stopping time with respect to $\mathcal{B}_n$;
\item $\E(\tau_{n}|\mathcal{B}_{n-1})=\E(|X'_{n}|^{2}|\mathcal{B}_{n-1}) $ a.s.;
\item for any $p>1$, there exists a constant $C_p<\infty$ depending only on $p$ such that
\[
\E(\tau_{n}^{p}|\mathcal{B}_{n-1})\leq C_{p}\E(|X'_{n}|^{2p}|\mathcal{B}_{n-1})=C_{p}\E(|X'_{n}|^{2p}|X'_1,\ldots,X'_{n-1}) \quad \hbox{a.s.},
\]
where $C_p=2(8/\pi^2)^{p-1}\Gamma(p+1)$, with $\Gamma$ being the usual Gamma function.
\end{enumerate}
\end{thm}

\begin{prop}\label{mpe}
Let $X_1, X_2,\ldots, X_n$ be real-valued random variables defined on a common probability space and $\|X_i\|_{p}<\infty$ for $1\le i\le n$, $p\ge 1$. Then
\[\Big\|\max_{1\le k\le n}| X_k|\Big\|_{p}\le n^{\frac{1}{p}}\max\{\|X_k\|_{p}: 1\le k\le n\}.\]
\end{prop}
\begin{proof}
We have $\max_{1\le k\le n}|X_k|^p\le \sum_{i=1}^{n}|X_i|^p$, and the proposition follows by taking expectation of both sides.
\end{proof}

\begin{thm}\cite{HH80}\label{bhi}
Let $X_1=S_1$, $X_i=S_{i}-S_{i-1}$ for $2\le i\le n$. If $\{S_i,\mathcal{F}_i,1\le i\le n\}$ is a martingale and $p>0$, then there exists a constant $C$ depending only on p such that
\[
\mathbb{E}\bigg(\max_{1\le i\le n}|S_i|^p\bigg)\le C\bigg\{\mathbb{E}\bigg[\bigg(\sum_{i=1}^{n}\mathbb{E}(X_{i}^2|\mathcal{F}_{i-1})\bigg)^{p/2}\bigg]+\mathbb{E}\bigg(\max_{1\le i\le n}|X_i|^p\bigg)\bigg\}.
\]
\end{thm}

\section*{Acknowledgements}
The authors are deeply grateful to the referees for their
great patience and very careful reading of the paper and for many valuable suggestions which lead to significant improvements of the paper. This work is supported by National Key R$\&$D Program of China (No. 2023YFA1009200), NSFC (Grants 11871132, 11925102), Liaoning Revitalization Talents Program (Grant XLYC2202042), and Dalian High-level Talent Innovation Program (Grant 2020RD09).


\end{document}